\documentclass{amsart}

\newcommand{\ta}{\tau}
\usepackage{graphicx} 

\usepackage{amsmath,amssymb,hyperref,mathabx,color,enumerate,chngpage}

\usepackage{mathrsfs}

\usepackage{lipsum}
\usepackage{amsfonts}
\usepackage{graphicx}
\usepackage{epstopdf}
\usepackage{algorithm}
\usepackage{algpseudocode}
\ifpdf
  \DeclareGraphicsExtensions{.eps,.pdf,.png,.jpg}
\else
  \DeclareGraphicsExtensions{.eps}
\fi

\newtheorem{definition}{Definition}[section]
\newtheorem{remark}{Remark}[section]
\newtheorem{theorem}{Theorem}[section]
\newtheorem{lemma}[theorem]{Lemma}

\newcommand{\algn}[1]{\begin{align} #1 \end{align}}
\newcommand{\algns}[1]{\begin{align*} #1 \end{align*}}

\newcommand{\mc}[1]{\mathcal{#1}}
\newcommand{\LRp}[1]{\left( #1 \right)}

\newcommand{\LRs}[1]{\left[ #1 \right]}

\newcommand{\dive}{\operatorname{div}}

\newcommand{\supp}{\operatorname{supp}}

\title[Optimal control problems of poroelasticity equations]{Distributed optimal control problems governed by poroelasticity equations}
\author{Arbaz Khan}
\address{Department of Mathematics, IIT Roorkee, India}
\email{arbaz@ma.iitr.ac.in}

\author{Jeonghun J. Lee}
\address{Department of Mathematics, Baylor University}
\email{Jeonghun\_Lee@baylor.edu}

\author{Harpal Singh}
\address{Department of Mathematics, IIT Roorkee, India}
\email{harpa@ma.iitr.ac.in}

\begin{document}


\begin{abstract}
	In this paper, we propose and analyze a novel two-field symmetric formulation with solid displacement and fluid pressure as main unknowns for the Biot’s consolidation model in poroelasticity. Firstly, we prove the well-posedness of the new formulation and then show the existence and uniqueness of optimal control where the fluid sources in the model act as a control variable. We prove a priori error estimates for the fully discrete scheme with backward Euler time discretization and a variational approximation of the control variable. A numerical example is presented to validate the performance of the proposed novel scheme.
\end{abstract}

\maketitle

\section{Introduction}
Mathematical models that describe how fluids move and how materials deform within porous structures are essential in many fields, such as geophysics, engineering, medicine and biological applications, such as in the development of the robust mechanical modelling of tissues and cells \cite{malandrino2019poroelasticity} and the research on the poroelastic component of the postseismic process \cite{mccormack2020modeling}. These models help us understand fluid flow in natural and artificial materials like cartilage, bones, tissue scaffolds, and even blood flow through organs like the brain, liver, and eyes. Because of their wide range of applications, flow problems in porous media have become a major topic of study in mathematics. Many researchers have focused on understanding the stability of these models, analyzing how sensitive they are to changes, and developing reliable numerical methods to simulate them. 

A key model in this area is Biot’s consolidation model \cite{biot1941general}. This model describes the interaction between interstitial fluid flowing through deformable porous media and has many applications, such as groundwater \cite{Kim1999549}, biological tissues \cite{Swan200325}, carbon sequestration \cite{WANGEN2016486}, and materials science \cite{MR2273503}. It plays a central role in simulating organ behaviour in medical applications and predicting how porous rocks respond in geophysical studies. In recent years, finite element methods have been widely developed and studied as powerful tools for solving Biot’s model accurately and efficiently.

\subsection{Biot’s consolidation model}
In this paper, we focus on quasi static consolidation problems, as discussed in \cite{MR3803860}. We consider situations where the consolidation process occurs slowly enough that acceleration effects can be neglected. Suppose that $\Omega \subset \mathbb{R}^d$, $d =2,3$ be a bounded domain with Lipschitz polygonal/polyhedral boundary $\partial \Omega$. We assume that there are two independent partitions of boundary $\partial \Omega$:
\begin{align*}
	\partial \Omega = \Gamma_p \cup  \Gamma_{f} = \Gamma_d \cup  \Gamma_t, \qquad \text{such that} \qquad  \Gamma_p \cap  \Gamma_{f} = \Gamma_d \cap  \Gamma_t = \emptyset,
\end{align*}
with positive $(n-1)$-dimensional surface measures $|\Gamma_p|, |\Gamma_d|, |\Gamma_t|>0$. The Biot’s consolidation model consists of an equilibrium equation, prescribing the conservation of the momentum, and of a continuity equation, stating the conservation of the mass. Biot’s consolidation model is described as follows: find the displacement of the porous medium $\boldsymbol{u}(t): \Omega \rightarrow \mathbb{R}^{d}$ and the fluid pressure $p(t): \Omega \rightarrow \mathbb{R}$ satisfying
\begin{subequations}
    \label{sys-eqs}
\begin{align}
	\label{sys1} - \text{div} (\boldsymbol{\mathcal{C}} \epsilon(\boldsymbol{u})) + \alpha \nabla p &= \tilde{\boldsymbol{f}} && \text{in} \ (0,T] \times \Omega,\\
	\label{sys2} s_{0} \dot{p} + \alpha \text{div}(\dot{\boldsymbol{u}}) -  \text{div} (\kappa \nabla p) &= -g && \text{in} \ (0,T] \times \Omega,
\end{align}
\end{subequations}
where $\mathcal{C}$ is the elastic stiffness tensor, $\epsilon(\boldsymbol{u}) = \frac{1}{2}\left(\nabla \boldsymbol{u} + \left(\nabla \boldsymbol{u}\right)^{T}\right)$ is the strain tensor, $\alpha >0$ is the Biot--Willis coefficient (usually close to 1), $\tilde{\boldsymbol{f}}(t): \Omega \rightarrow \mathbb{R}^{d}$ denotes the body force acting on the solid structure, $s_{0} \ge 0$ is the constrained specific storage coefficient, the dot over the variables represents time derivative, $\kappa$ is the permeability tensor which is a $d\times d$ symmetric positive definite matrix with constant entries, and $g(t): \Omega \rightarrow \mathbb{R}$ is the source or sink term representing fluid gain or loss. For isotropic elastic porous media, $\boldsymbol{\mathcal{C}} \epsilon(\boldsymbol{u}) = 2\mu  \epsilon(\boldsymbol{u}) + \lambda \text{div}(\boldsymbol{u}) \boldsymbol{I}$,
where $\boldsymbol{I}\in \mathbb{R}^{d\times d}$ is the identity matrix, $\lambda$ and $\mu$ are Lam\'{e} constants related to the Young's elasticity modulus $E$ and the Poisson’s ratio $0<\nu<1/2$ as
\begin{align*}
	\lambda = \frac{\nu E}{(1-2\nu)(1+\nu)},\qquad \qquad \mu = \frac{E}{2(1+\nu)}.
\end{align*}
The boundary conditions are given by
\begin{align*}
	p(t) &=0 \quad \text{on} \ \Gamma_p, & -\kappa \nabla p(t) \cdot \boldsymbol{n}  &= 0 \quad \text{on} \ \Gamma_{f},
    \\
	\boldsymbol{u}(t) &=0 \quad \text{on} \ \Gamma_d, & \ \ (\mathcal{C} \epsilon(\boldsymbol{u}) - \alpha p \mathbb{I}) \boldsymbol{n} &= \boldsymbol{0} \quad \text{on} \ \Gamma_t,
\end{align*}
for all $t \in (0,T]$, where $\boldsymbol{n}$ is the outward unit normal vector field on $\partial \Omega$. The boundary conditions on $\Gamma_p$ and $\Gamma_{f}$ are the Dirichlet and Neumann boundary conditions in the Darcy flow problems and those on $\Gamma_d$ and $\Gamma_t$ are also Dirichlet and Neumann boundary conditions in linear elasticity. In the rest of this paper we assume that all boundary conditions are homogeneous for simplicity of presentation.

The poroelastic locking typically occurs when $s_0 \ge 0$ is very close to $0$, which corresponds to the case that the fluid and the elastic medium are (almost) incompressible \cite{phillips2009}.
In the literature, several numerical approaches with different formulations have been proposed to overcome poroelasticity locking. For instance, the Galerkin least-squares method based on a four-field formulation (with displacement, stress, fluid flux, and pressure as unknowns) was studied in \cite{MR2177147}. The mixed and discontinuous Galerkin (dG) finite element methods that successfully avoid pressure oscillations were developed in \cite{phillips2008coupling}. A dG method for the two-field formulation was presented in \cite{MR3047799}, while a locking-free non-symmetric interior penalty dG method was introduced in \cite{MR3606362}. Moreover, non-conforming and stabilized finite element methods have been explored to address locking issues in \cite{MR3803860, MR3590654}. More recently, hybridizable discontinuous Galerkin (HDG) methods have been applied to the Biot’s consolidation model \cite{MR3907413}. An embedded-hybridizable DG method was developed for the total pressure formulation in \cite{MR4659441}, a locking-free three-field virtual element method (VEM) in \cite{MR4221326}, a non-conforming locking-free VEM was proposed in \cite{MR4636155}, 
 a non-conforming hybrid high-order method on general meshes in \cite{MR3504993}, and a non-symmetric approach with a quasi-optimal and robust discretization in \cite{MR4405491}.

Optimal control problems arise naturally in many biological and biomechanical applications, where controlling physical processes is essential for understanding and influencing systems. One prominent example is the optimization of flow pressure and the analysis of how key biological parameters influence and can be controlled within the lamina cribrosa--a porous connective tissue located at the base of the optic nerve head in the eye. This tissue is commonly modelled using poroelasticity, and variations in pressure and material parameters are widely believed to play a significant role in the onset and progression of ocular neurodegenerative disorders, most notably glaucoma. Despite the relevance of such models, the literature on optimal control problems governed by poroelastic systems remains rather limited. In this direction, Bociu et al. \cite{MR4410836} studied a linear–quadratic elliptic–parabolic optimal control problem describing fluid flow in deformable porous media, under the assumption of a vanishing specific storage coefficient $s_{0} = 0$. They established the existence and uniqueness of optimal controls and derived the associated first-order necessary optimality conditions.


\subsection{Primary contributions}
To the best of our knowledge, the numerical approximation of optimal control problems governed by these models has not yet been investigated in the existing literature, and the present work constitutes a first step in this direction. Our main goal is to control the system in such a way that the displacement of the solid structure and the fluid pressure remain close to some desired target values. We treat the fluid sources in the model as a control variable. Next, we highlight the main contributions of this work:
\begin{itemize}
	\item \textit{Novel symmetric formulation:} We propose a novel two-field symmetric formulation with displacement and pore pressure as unknowns to deal with Biot's consolidation model.
	\item \textit{Existence and uniqueness of optimal control:} We establish the well-posedness of the new symmetric formulation and then prove the existence and uniqueness of an optimal control (see Theorem~\ref{existence}) and derive the first order necessary and sufficient optimality conditions..
	\item \textit{A priori error estimates for a fully discrete scheme:} We prove the a priori error estimates for a fully discrete scheme using the backward Euler time discretization. In particular, the estimates for displacement and pressure variables are derived in $L^{\infty}(0,T;H^{1}(\Omega))$ norm and in $L^{2}(0,T;L^{2}(\Omega))$ norm for the control variable which is approximated using a \textit{variational discretization} approach proposed by Hinze in \cite{MR2122182}, where no explicit discretization of the control variable is used and the discrete control is achieved by projecting the discrete adjoint state on the admissible control set.
\end{itemize}
\subsection{Organization of the paper}
The rest of the paper is organized as follows: In Section~\ref{Model problem and its well-posedness}, we show the existence of weak solutions and then the existence of optimal control combined with the optimality conditions. In Section~\ref{Discrete formulation}, we introduce the finite element approximation and a fully discrete formulation for the optimality system. Section~\ref{A priori Error analysis} is dedicated to the derivation of a priori error estimates in the $L^{\infty}(0,T;H^1(\Omega))$ norm for the state and adjoint variables, and in the $L^{2}(0,T;L^2(\Omega))$ norm for the optimal control using a \textit{variational discretization} approach. 
\section{Well-posedness}\label{Model problem and its well-posedness}
In this section, we discuss the well--posedness of the model problem,along with existence of optimal control and the optimality conditions. Firstly, we introduce some notations that we are going to use throughout.
\vspace{-1mm}
\subsection{Notations}
For a Banach space $X$ with norm $\|\cdot \|_X$ we will use $L^r(0,t; X)$, $C^0(0,t; X)$ to denote \vspace{-2mm}
\begin{align*}
	\| f \|_{L^r(0,t;X)} &:= \left( \int_0^t \| f(s) \|_X^r \,ds \right)^{1/r} 
	\qquad \text{and} \qquad
	\| f \|_{C^0(0,t;X)} := \sup_{0\le s\le t} \| f(s) \|_X. 
\end{align*}
For functions $v, w \in L^2(D)$ let
\begin{align*}
    (v, w)_{D} := \int_{D} v w \,dx ,
\end{align*}
and this definition is naturally extended to vector-valued or tensor-valued $L^2(D)$ functions $v, w$. 
We consider the following spaces:
\begin{align}
	\boldsymbol{V}&= \{\boldsymbol{v} \in H^1(\Omega; \mathbb{R}^d) \,:\, \boldsymbol{v}|_{\Gamma_d} = 0\},\qquad
	Q= \{ q \in H^1(\Omega) \,:\, q|_{\Gamma_p} = 0 \}.     
\end{align}
We use $\boldsymbol{L}^2(\Omega)$ and $\boldsymbol{H}^1(\Omega)$ to emphasize function spaces of vector-valued functions. However, we do not use boldface symbols to denote norms of vector-valued functions, i.e., $\| \cdot \|_{L^2(\Omega)}$ will be used for the $L^2(\Omega)$ norms of scalar- and vector-valued functions.

To formulate the optimal control problem, we introduce the admissible control set $\mathcal{M}_{ad}$ which is defined as:
\begin{align}\label{admcontrol}
	\mathcal{M}_{ad}:= \{m \in L^{2}(0,T;L^{2}(\Omega)): m_{a}\le m(t,x) \le m_b \ a.e. \ \text{in} \ (0,T) \times \Omega\},
\end{align}
where the bounds $m_{a}, m_{b} \in \mathbb{R}$ such that $m_{a} < m_{b}$.
\subsection{A priori estimates}\label{A priori estimates}
First we consider the following reformulation of variational equations by taking the time derivative of \eqref{sys1}:
\begin{subequations}
	\label{eq:new-strong-eqs}
\begin{align}
	\label{eq:new-strong-eq1} - \text{div} (\boldsymbol{\mathcal{C}} \epsilon(\dot{\boldsymbol{u}})) + \alpha \nabla \dot{p} &= {\boldsymbol{f}} && \text{in} \ (0,T] \times \Omega,\\
	\label{eq:new-strong-eq2} s_{0} \dot{p} + \alpha \text{div}(\dot{\boldsymbol{u}}) -  \text{div} (\kappa \nabla p) &= -g && \text{in} \ (0,T] \times \Omega 
\end{align}
\end{subequations}
where $\boldsymbol{f}$ is the time derivative of $\tilde{\boldsymbol{f}}$. 
Compared to the original system \eqref{sys-eqs} which has an algebraic equation \eqref{sys1}, this new formulation is a system of differential equations. In contrast to the standard theory of evolutionary partial differential equations which needs initial data of $\boldsymbol{u}$ and $p$. However, the time derivative involved part is symmetric but not positive definite, so we cannot apply a standard well-posedness theory of PDEs to this system. In the discussions below we show existence and uniqueness of weak solutions of \eqref{eq:new-strong-eqs}. 
\begin{remark}
    We remark that a solution of \eqref{eq:new-strong-eqs} becomes a solution of \eqref{sys-eqs} if initial data $(\boldsymbol{u}(0), p(0))$ satisfies \eqref{sys1} at $t=0$, so the solution also solves the original system.   
\end{remark}

The bilinear forms $a_{\boldsymbol{u}}:\boldsymbol{V} \times \boldsymbol{V} \rightarrow \mathbb{R}$, $b:Q \times \boldsymbol{V} \rightarrow \mathbb{R}$, $a_p:Q \times Q \to \mathbb{R}$ are defined as
	\begin{align*}
		a_{\boldsymbol{u}}(\boldsymbol{v}, \boldsymbol{w}) &:= \int_{\Omega} \boldsymbol{\mathcal{C}} \epsilon(\boldsymbol{v}) : \epsilon(\boldsymbol{w}) \,dx = 2\mu \int_{\Omega} \epsilon(\boldsymbol{v}) : \epsilon(\boldsymbol{w}) \,dx + \lambda \int_{\Omega} \dive( \boldsymbol{v}) \dive( \boldsymbol{w}) \,dx,
        \\
		b(q,\boldsymbol{v}) &:= - \int_{\Omega} \alpha q \dive( \boldsymbol{v}) \,dx, 
        \qquad a_p(q,r) := \int_{\Omega} \kappa \nabla q \cdot \nabla r \,dx.
	\end{align*}
We set 
\begin{align}
    \| \boldsymbol{v}\|_{a_{\boldsymbol{u}}} := (a_{\boldsymbol{u}} (\boldsymbol{v}, \boldsymbol{v}))^{\frac 12}, \qquad     \| q\|_{a_{p}} := (a_{p} (q, q))^{\frac 12}.
\end{align}
%
It is known that there exists $C>0$ such that 
\begin{align}
	\label{eq:continuous-inf-sup}
	\inf_{0\not = q \in L^2(\Omega) } \sup_{\boldsymbol{v}\in \boldsymbol{V}} \frac{ (q,  \dive( \boldsymbol{v}))_{\Omega} }{{\|\boldsymbol{v}\|_{\boldsymbol{V}}\|q\|_{L^2(\Omega)}}} \ge C.
\end{align}
Suppose that $\boldsymbol{V}_n \subset \boldsymbol{V}$, $Q_n \subset Q$ are finite dimensional subspaces such that they satisfy the Ladyzhenskaya--Babu\v{s}ka--Brezzi (LBB) condition 
%
\begin{align}
	\label{eq:galerkin-inf-sup}
	\inf_{0\not = q \in Q_n } \sup_{\boldsymbol{v}\in \boldsymbol{V}_n} \frac{ (q,  \dive( \boldsymbol{v}))_{\Omega} }{{\|\boldsymbol{v}\|_{H_x^1}\|q\|_{L^2(\Omega)}}} \ge C >0
\end{align}
with $C$ independent of $n$. 
It implies that there exists a linear map $R_n:Q_n \to \boldsymbol{V}_n$ such that $b(q, R_n q) = \alpha \|q\|_{L^2(\Omega)}^2$ and $\|R_n q\|_{H^1(\Omega)} \le C \|q\|_{L^2(\Omega)}$ hold for all $q \in Q_n$ with $C>0$. By a standard argument in the theory of saddle point problems, one can show from \eqref{eq:galerkin-inf-sup} that there exist constants $\delta_0, C_0>0$ satisfying 
\begin{align}
	\label{eq:Babuska-Aziz}
	a_{\boldsymbol{u}} (\boldsymbol{v}, \boldsymbol{v} + \delta_0 R_n q) + b(q,\boldsymbol{v} + \delta_0 R_n q)  - b(q,\boldsymbol{v}) + (s_0 q, q)_{\Omega} \ge C_0 (\|\boldsymbol{v}\|_{\boldsymbol{V}}^2 + \|q\|_{L^2(\Omega)}^2)
\end{align}
for any $(\boldsymbol{v}, q) \in \boldsymbol{V}_n \times Q_n$, and the constants $\delta_0, C_0$ have lower bounds only depend on the inf-sup constant \eqref{eq:galerkin-inf-sup}, the boundedness constants of $a_{\boldsymbol{u}}(\cdot, \cdot)$ and $b(\cdot, \cdot)$, and the coercivity constant of $a_{\boldsymbol{u}}(\cdot, \cdot)$. 
We consider the following variational equations derived from \eqref{eq:new-strong-eqs}: 
\begin{subequations}\label{eq:new-weak-n-eqs}
	\begin{align}
		\label{eq:new-weak-n-eq1} 
        a_{\boldsymbol{u}} (\dot{\boldsymbol{u}}(t), \boldsymbol{v}) + b(\dot{p}(t),\boldsymbol{v}) &= (\boldsymbol{f}(t), \boldsymbol{v})_{\Omega}, \\
		\label{eq:new-weak-n-eq2} 
		b(q,\dot{\boldsymbol{u}}(t)) - (s_0 \dot{p}(t), q)_{\Omega} - a_p(p(t), q) &= (g(t), q)_{\Omega}
	\end{align}
\end{subequations}
for all $(\boldsymbol{v}, q) \in \boldsymbol{V}_n \times Q_n$ and for all $t >0$. This system is a well-posed linear ODE system because the mass matrix is invertible due to \eqref{eq:galerkin-inf-sup}. 

For later use, we assume the following elliptic regularity property for the solutions of the differentiated poroelastic system. There exists \(0<s\le 1\) such that, whenever the data are sufficiently regular, the solution satisfies
\begin{multline}
    \label{eq:elliptic-regularity}
    \|p\|_{L^2(0,T;H^{1+s}(\Omega))}
 + \|\boldsymbol u\|_{L^2(0,T;\boldsymbol H^{1+s}(\Omega))}
    \\
    \le C_{\rm reg}
 \left(
 \|\boldsymbol u (0)\|_{\boldsymbol H^{1+s}(\Omega)}
 + \|p (0)\|_{H^1(\Omega)}
 + \|\boldsymbol f\|_{L^2(0,T;\boldsymbol L^2(\Omega))}
 + \|g\|_{L^2(0,T;L^2(\Omega))}
 \right).
\end{multline}
%
\begin{lemma}
	Suppose that $(\boldsymbol{u}, p) \in C^0([0,T]; \boldsymbol{V}_n \times Q_n) \cap C^1(0,T; \boldsymbol{V}_n \times Q_n)$ satisfy \eqref{eq:new-weak-n-eqs} for all $0< t <T$. Then, 
\algn{
	\label{eq:energy-estimate2}
	&\| \dot{\boldsymbol{u}} \|_{L^2(0,t; \boldsymbol{V})} + \|\dot{p}\|_{L^2(0,t; L^2(\Omega))} + \|p\|_{C^0(0,t; Q)} 
	\\
	\notag
	&\quad \le C (\|p(0)\|_{Q} + \|\boldsymbol{f}\|_{L^2(0,t; L^2(\Omega))} + \|g\|_{L^2(0,t;L^2(\Omega))}), 
	\\
	\label{eq:energy-estimate1}
	&\|{\boldsymbol{u}}\|_{C^0(0,t;\boldsymbol{V})} + \|{p}\|_{C^0(0,t;L^2(\Omega))} 
	\\
	\notag
	&\quad \le \|\boldsymbol{u}(0)\|_{\boldsymbol{V}} + \|p(0)\|_{L^2(\Omega)} + C\sqrt{t} (\|p(0)\|_{Q} + \|\boldsymbol{f}\|_{L^2(0,t; L^2(\Omega))} + \|g\|_{L^2(0,t;L^2(\Omega))} ) .	
}
\end{lemma}
\begin{proof}
Let $\boldsymbol{v} = \dot{\boldsymbol{u}} + \delta_0 R_n \dot{p}$ and $q = -\dot{p}$ in \eqref{eq:new-weak-n-eq1} and \eqref{eq:new-weak-n-eq2}. If we add the equations and integrate from 0 to $t$, then 
\begin{align*}
	&\int_0^t \LRs{ a_{\boldsymbol{u}}(\dot{\boldsymbol{u}}(s), \dot{\boldsymbol{u}}(s) + \delta_0 R_n \dot{p}(s)) + \delta_0 b(\dot{p}(s), R_n \dot{p}(s)) + (s_0 \dot{p}(s), \dot{p}(s))_{\Omega} } \,ds + \frac 12 a_p(p(t), p(t))
	\\
	&= \frac 12 a_p( p(0), p(0)) + \int_0^t (\boldsymbol{f}(s) , \dot{\boldsymbol{u}}(s) + \delta_0 R_n \dot{p}(s) )_{\Omega} \,ds - \int_0^t ( g(s), \dot{p}(s))_{\Omega} \,ds .
\end{align*}
By \eqref{eq:Babuska-Aziz}, the Cauchy--Schwarz inequality, and the boundedness of $R_n$, we have
\begin{align*}
	&C_0 (\| \dot{\boldsymbol{u}} \|_{L^2(0,t; \boldsymbol{V})}^2 + \| \dot{p} \|_{L^2(0,t; L^2(\Omega))}^2 ) + \frac 12 \| p(t) \|_Q^2
	\\
	&\le \frac 12 \| p(0) \|_Q^2 + C \|\boldsymbol{f}\|_{L^2(0,t; L^2(\Omega))} (\|\dot{\boldsymbol{u}}\|_{L^2(0,t; \boldsymbol{V})} + \|\dot{p}\|_{L^2(0,t;L^2(\Omega))})  + \|g\|_{L^2(0,t; L^2(\Omega))} \|\dot{p}\|_{L^2(0,t;L^2(\Omega))}.
\end{align*}
Then, \eqref{eq:energy-estimate2} follows by Young's inequality.

The fundamental theorem of calculus leads to
\begin{align*}
	\|\boldsymbol{u}(t)\|_{\boldsymbol{V}} + \|p(t)\|_{L^2(\Omega)} &\le \|\boldsymbol{u}(0)\|_{\boldsymbol{V}} + \|p(0)\|_{L^2(\Omega)} + \|\dot{\boldsymbol{u}}\|_{L^1(0,t;\boldsymbol{V})} + \|\dot{p}\|_{L^1(0,t;L^2(\Omega))} 
	\\
	&\le \|\boldsymbol{u}(0)\|_{\boldsymbol{V}} + \|p(0)\|_{L^2(\Omega)} + \sqrt{t} (\|\dot{\boldsymbol{u}}\|_{L^2(0,t;\boldsymbol{V})} + \|\dot{p}\|_{L^2(0,t;L^2(\Omega))}),
\end{align*}
so this estimate and \eqref{eq:energy-estimate2} yield \eqref{eq:energy-estimate1}. 
\end{proof}
%
\subsection{Existence of weak solutions}\label{Existence of weak solutions}
\begin{definition} 
	For given $(\boldsymbol{u}_0, p_0) \in \boldsymbol{V} \times Q$ and $T>0$, we define that 
	$\LRp{\boldsymbol{u}, p}\in H^1(0,T;\boldsymbol{V}) \times (H^1(0,T;L^2(\Omega))\cap L^2(0,T;Q))$ is called a weak solution of \eqref{eq:new-strong-eqs} if 
	\begin{subequations}
		\label{eq:weak-form-eqs}
	\begin{align}
		\label{eq:weak-form-eq1}
		\int_0^T \LRp{a_{\boldsymbol{u}} \LRp{\dot{\boldsymbol{u}} (t), \boldsymbol{v} (t)} + b \LRp{\dot{p} (t), \boldsymbol{v} (t)}} \, dt &= \int_0^T \LRp{\boldsymbol{f} (t), \boldsymbol{v} (t)}_{\Omega} \, dt, \\
		\label{eq:weak-form-eq2}
		\int_0^T \LRp{b \LRp{q(t), \dot{\boldsymbol{u}} (t)} - \LRp{s_0 \dot{p} (t), q(t)}_{\Omega} - a_p\LRp{ p(t) , q(t)}} \, dt &= \int_0^T \LRp{g(t), q(t)}_{\Omega} \, dt, \\
		\label{eq:weak-form-IC}
		\boldsymbol{u} (0) &= \boldsymbol{u}_0 \text{, } p(0) = p_0
	\end{align}
	\end{subequations}
	holds for all $\LRp{\tilde{\boldsymbol{v}}, \tilde{q}} \in C_0^{\infty} \LRp{[0, T); \boldsymbol{V} \times Q}$ where 
\algn{
	\label{eq:spacetime-test-function}
	&C_0^{\infty} \LRp{[0, T); \boldsymbol{V} \times Q} \\
	\nonumber &:=\{ (\boldsymbol{v},q) \in C^{\infty}([0,T); \boldsymbol{V} \times Q) \,:\, (\supp \boldsymbol{v} \cup \supp q) \subset \subset [0,T) \times \Omega \}.
}
\end{definition}
In the discussions below, we use $L_t^rX$ to denote $L^r(0,T;X)$ for fixed $T>0$. 
\begin{theorem}
	For given $(\boldsymbol{u}_0, p_0) \in \boldsymbol{V} \times Q$ and $T>0$ there exists a unique weak solution $(\boldsymbol{u}, p)$ of \eqref{eq:new-strong-eqs} which satisfies \eqref{eq:energy-estimate2}, \eqref{eq:energy-estimate1}. Under the regularity assumption \eqref{eq:elliptic-regularity}, the solution satisfies higher regularity conditions in \eqref{eq:elliptic-regularity}.
%
\end{theorem}
\begin{proof}
The estimates \eqref{eq:energy-estimate2}, \eqref{eq:energy-estimate1} imply uniqueness of solutions, so we only focus on existence.

Let $\{ \boldsymbol{V}_n \times Q_n \}_{n=1}^{\infty}$ be a nested sequence of finite dimensional subspaces of $\boldsymbol{V} \times Q$ such that $\boldsymbol{V}_{n} \times Q_{n} \subset \boldsymbol{V}_{n+1} \times Q_{n+1}$ and $\bigcup_{n=1}^{\infty} \boldsymbol{V}_n$, $\bigcup_{n=1}^{\infty} Q_n$ are dense in $\boldsymbol{V}$, $Q$, respectively, with the norms $\|\cdot\|_{\boldsymbol{V}}$, $\|\cdot \|_Q$. We also assume that there exists $C>0$ independent of $n$ such that 
\algn{
	\label{eq:discrete-inf-sup}
	\inf_{0\not = q \in Q_n} \sup_{\boldsymbol{v}\in \boldsymbol{V}_n} \frac{ ( q, \dive (\boldsymbol{v}))_{\Omega}}{\|\boldsymbol{v}\|_{\boldsymbol{V}}\|q\|_{L^2(\Omega)}} \ge C .
}
Given $\LRp{\boldsymbol{u}(0), p(0)} \in \boldsymbol{V} \times Q$ let $\{ \LRp{\boldsymbol{u}_{n, 0}, p_{n, 0}} \in \boldsymbol{V}_n \times Q_n \}_{n=1}^{\infty}$ be a sequence such that
\algn{
	\nonumber \|\boldsymbol{u}_{n, 0} - \boldsymbol{u}_0\|_{\boldsymbol{V}} + \|p_{n, 0} - p_0\|_{Q} \to 0 \text{ as }  n \to \infty. 
}
Let $\boldsymbol{f}_n$, $g_n$ be the $L^2$ projections of $\boldsymbol{f}$, $g$ to $\boldsymbol{V}_n$ and $Q_n$. For each $n$ consider the problem finding $\LRp{\boldsymbol{u}_n, p_n} \in C^0 \LRp{[0, T]; \boldsymbol{V}_n \times Q_n} \cap C^1 \LRp{0, T; \boldsymbol{V}_n \times Q_n}$ satisfying \eqref{eq:new-weak-n-eqs} 
%
with initial data $\boldsymbol{u}_n(0) = \boldsymbol{u}_{n, 0}$, $p_n (0) = p_{n, 0}$. 
Then, by the energy estimates \eqref{eq:energy-estimate2} and \eqref{eq:energy-estimate1}, 
\algn{
	\nonumber &\|\boldsymbol{u}_n\|_{L_t^{\infty} \boldsymbol{V}} + \|\dot{\boldsymbol{u}}_n\|_{L_t^{2} \boldsymbol{V}} + \|p_n\|_{L_t^{\infty} Q} + \|\dot{p}_n\|_{L_t^{2} L^2(\Omega)} \\
	\nonumber &\le C (\|\boldsymbol{u}(0)\|_{\boldsymbol{V}} + \|p(0)\|_{Q} + \|\boldsymbol{f}\|_{L_t^2L^2(\Omega)} + \|g\|_{L_t^2 L^2(\Omega)}) , 
}
so $\{ \boldsymbol{u}_n \}$, $\{p_n\}$ are bounded sequences in $L_t^{\infty} \boldsymbol{V} \cap L_t^{2} \boldsymbol{V}$ and $L_t^{\infty} Q \cap L_t^{2} L^2(\Omega)$, respectively. By Banach-Alaoglu theorem, there is a subsequence $\{\LRp{\boldsymbol{u}_{n_i}, p_{n_i}} \}_{i=1}^{\infty}$ such that
\algns{
	\boldsymbol{u}_{n_i} &\stackrel{*}{\rightharpoonup} \boldsymbol{w}_0 \in L_t^{\infty} \boldsymbol{V}, \quad \dot{\boldsymbol{u}}_{n_i} \rightharpoonup \boldsymbol{w}_1 \in L_t^2 \boldsymbol{V}, \quad p_{n_i} \stackrel{*}{\rightharpoonup} r_0 \in L_t^{\infty} Q, \quad \dot{p}_{n_i} \rightharpoonup r_1 \in L_t^2 L^2(\Omega)
}
as $i \to \infty$. Without loss of generality, we may assume that $\{(\boldsymbol{u}_n, p_n)\}_{n=1}^{\infty}$ is such a subsequence itself.

We can show that $\boldsymbol{w}_0$ has a weak time derivative, and the weak time derivative of $\boldsymbol{w}_0$ is $\boldsymbol{w}_1$, by 
\begin{align*}
    -\int_0^T \langle \boldsymbol{w}_0, \boldsymbol{z} \rangle_{\boldsymbol{V}, \boldsymbol{V}^*} \phi'(t)\,dt &= -\lim_{n \to \infty} \int_0^T \langle \boldsymbol{u}_n, \boldsymbol{z} \rangle_{\boldsymbol{V}, \boldsymbol{V}^*} \phi'(t)\,dt 
    \\
    &= \lim_{n \to \infty} \int_0^T \langle \dot{\boldsymbol{u}}_n, \boldsymbol{z} \rangle_{\boldsymbol{V}, \boldsymbol{V}^*} \phi(t)\,dt 
    \\
    &= \int_0^T \langle \boldsymbol{w}_1, \boldsymbol{z} \rangle_{\boldsymbol{V}, \boldsymbol{V}^*} \phi(t)\,dt ,
\end{align*}
for any $\boldsymbol{z} \in \boldsymbol{V}^*$ and a compactly supported smooth function $\phi(t)$ on $(0,T)$. Similarly, we can show that the weak derivative of $r_0$ is $r_1$. 

By integration in time,
\begin{align*}
	&\int_0^T ( a_{\boldsymbol{u}} \LRp{\dot{\boldsymbol{u}}_n (s), \boldsymbol{v} (s)} + b \LRp{\dot{p}_n (s), \boldsymbol{v} (s)} + b \LRp{q(s), \dot{\boldsymbol{u}}_n (s)} \\ \notag
	&- \LRp{s_0 \dot{p}_n (s), q (s)}_{\Omega} - a_p\LRp{ p_n (s), q (s)} ) \, ds = \int_0^T \LRp{\LRp{\boldsymbol{f}_n (s), \boldsymbol{v} (s)}_{\Omega} + \LRp{g_n (s), q (s)}_{\Omega}} \, ds
\end{align*}
holds for all $\LRp{\boldsymbol{v}, q} \in C_0^{\infty} \LRp{[0, T); \boldsymbol{V}_n \times Q_n}$. 
For simplicity we define $B_0 \LRp{\LRp{\boldsymbol{v}, q}, \LRp{\tilde{\boldsymbol{v}}, \tilde{q}}}$ for $\boldsymbol{v}, \tilde{\boldsymbol{v}} \in \boldsymbol{V}$ and $q, \tilde{q} \in Q$ by
\algn{
	\nonumber B_0 \LRp{\LRp{\boldsymbol{v}, q}, (\tilde{\boldsymbol{v}}, \tilde{q}) }
	= a_{\boldsymbol{u}} ({\boldsymbol{v}, \tilde{\boldsymbol{v}}}) + b ({q, \tilde{\boldsymbol{v}}}) + b \LRp{\tilde{q}, \boldsymbol{v}} - \LRp{s_0 q, \tilde{q}}_{\Omega}.
}
The previous formula is
\algn{
	\label{eq:star-1} \int_0^T &\LRp{B_0 \LRp{\LRp{\dot{\boldsymbol{u}}_n (s), \dot{p}_n (s)}, \LRp{\boldsymbol{v} (s), q (s)}} - a_p\LRp{ p_n (s), q (s)}} \, ds \\ \notag
	&= \int_0^T \LRp{\LRp{\boldsymbol{f}_n (s), \boldsymbol{v} (s)}_{\Omega} + \LRp{g_n (s), q (s)}_{\Omega}} \, ds .
}
Since $\dot{\boldsymbol{u}}_n \rightharpoonup \boldsymbol{w}_1$ in $L_t^2 \boldsymbol{V}$, $\dot{p}_n \rightharpoonup r_1$ in $L_t^2 L^2(\Omega)$, $p_n \rightharpoonup r_0$ in $L_t^2 Q$, $\boldsymbol{f}_n \to \boldsymbol{f}$ in $L_t^2 \boldsymbol{V}^*$, $g_n \to g$ in $L_t^2 Q^*$,
\algn{
	\label{eq:star-2} \int_0^T &\LRp{B_0 \LRp{\LRp{\boldsymbol{w}_1 (s), r_1 (s)}, \LRp{\boldsymbol{v} (s), q (s)}} - a_p\LRp{ r_0 (s), q (s)}} \, ds \\ \notag 
	&= \int_0^T \LRp{\LRp{\boldsymbol{f} (s), \boldsymbol{v} (s)}_{\Omega} + \LRp{g (s), q (s)}_{\Omega}} \, ds .
}
Since $\bigcup_{n=1}^{\infty} \boldsymbol{V}_n$ and $\bigcup_{n=1}^{\infty} Q_n$ are dense in $\boldsymbol{V}$ and $Q$, respectively, \eqref{eq:star-2} holds for all $\LRp{\boldsymbol{v}, q} \in C_0^{\infty}([0,T);\boldsymbol{V} \times Q)$ where $C_0^{\infty}([0,T);\boldsymbol{V} \times Q)$ is defined as in \eqref{eq:spacetime-test-function}.

It remains to verify the initial conditions. Since
\[
\boldsymbol u_n(t)=\boldsymbol u_{n,0}
+\int_0^t \dot{\boldsymbol u}_n(s)\,ds,
\qquad
p_n(t)=p_{n,0}+\int_0^t \dot p_n(s)\,ds,
\]
passing to the limit in \(C^0([0,T];\boldsymbol V^*)\) and
\(C^0([0,T];L^2(\Omega))\), respectively, and using
\(\boldsymbol u_{n,0}\to \boldsymbol u_0\) in \(\boldsymbol V\) and
\(p_{n,0}\to p_0\) in \(Q\), we obtain
\(\boldsymbol w_0(0)=\boldsymbol u_0\) and \(r_0(0)=p_0\).
\end{proof}

\subsection{Existence of optimal control}\label{Existence of optimal control}
Our main goal is to choose a distributed control $m \in \mc{M}_{ad}$ in such a way that the corresponding solid displacement $\boldsymbol{u}$ and fluid pressure $p$ are the best possible approximations to the desired/target values (dictated by the applications), denoted by $\boldsymbol{u}_C$ and $p_C$, respectively.
\begin{definition} 
	For given $\boldsymbol{u}_0 \in \boldsymbol{V}$ and $p_0 \in Q$, we define that 
	$\LRp{\boldsymbol{u}, p, m}\in H^1(0,T;\boldsymbol{V}) \times (H^1(0,T;L^2(\Omega))\cap L^2(0,T;Q)) \times L^2(0,T;L^2(\Omega))$ is admissible  for initial data $(\boldsymbol{u}_0, p_0)$ if $\LRp{\boldsymbol{u}, p, m}$ satisfies
	\algn{
		\nonumber \int_0^T \LRp{a_{\boldsymbol{u}} \LRp{\dot{\boldsymbol{u}} (t), \boldsymbol{v} (t)} + b \LRp{\dot{p} (t), \boldsymbol{v} (t)}} \, dt &= \int_0^T \LRp{\boldsymbol{f} (t), \boldsymbol{v} (t)}_{\Omega} \, dt, \\
		\nonumber \int_0^T \LRp{b \LRp{q(t), \dot{\boldsymbol{u}} (t)} - \LRp{s_0 \dot{p} (t), q(t)}_{\Omega} - a_p\LRp{p(t), q(t)}} \, dt &= \int_0^T \LRp{m(t), q(t)}_{\Omega} \, dt, \\
		\nonumber \boldsymbol{u} (0) &= \boldsymbol{u}_0 \text{, } p(0) = p_0.
	}
\end{definition}
Let $\alpha_{\boldsymbol{u}, C}$ and $\alpha_{p, C}$ be non-negative numbers such that $ \alpha_{\boldsymbol{u}, C} +  \alpha_{p, C} > 0$, and the constant $\gamma > 0$ is a control cost parameter. For given target velocity and pressure fields 
\begin{align*}
    \boldsymbol{u}_C \in L^2(0,T; \boldsymbol{L}^2(\Omega)), \quad p_C \in L^2(0,T; L^2(\Omega)), 
\end{align*}
and an admissible pair $\LRp{\boldsymbol{u}, p, m}$, we define the objective functional $J \LRp{\boldsymbol{u}, p, m}$ as
\begin{align}
	\label{obj_fun} J \LRp{\boldsymbol{u}, p, m} = 
	  \frac{\alpha_{\boldsymbol{u}, C}}{2} \int_0^T \int_{\Omega} |\boldsymbol{u} \LRp{x, t} \hspace{-0.2mm}-\hspace{-0.2mm} \boldsymbol{u}_C \LRp{x, t}|^2 \, dx \, dt 
     \\
	 + \frac{\alpha_{p, C}}{2} \int_0^T \int_{\Omega} |p(x, t) - p_C (x, t)|^2 \, dx \, dt + \frac{\gamma}{2} \int_0^T \int_{\Omega} |m(x, t)|^2 \, dx \, dt \notag,
\end{align}
where $\mc{M}_{ad}$ is the closed and convex admissible set of control variable defined as in \eqref{admcontrol}. The term $\frac{\gamma}{2} \int_0^T \int_{\Omega} |m(x, t)|^2 \, dx \, dt$ is introduced to bound the control function and to prove the existence of an optimal control. 
\begin{definition}
	$\bar{m} \in \mc{M}_{ad}$ is called an optimal control if the admissible triple $\LRp{\bar{\boldsymbol{u}}, \bar{p}, \bar{m}}$ for given initial data $(\boldsymbol{u}_0, p_0)$ satisfies
	\algn{
		\nonumber J \LRp{\boldsymbol{u}, p, m} \ge J \LRp{\bar{\boldsymbol{u}}, \bar{p}, \bar{m}}
	}
	for all $\LRp{\boldsymbol{u}, p, m}$ admissible for given initial data $(\boldsymbol{u}_0, p_0)$ and for $m \in \mc{M}_{ad}$.
\end{definition}
\begin{theorem} \label{existence}
	For any given initial data $(\boldsymbol{u}_0, p_0)$ there exists a unique optimal control $\bar{m} \in \mc{M}_{ad}$.
\end{theorem}
\begin{proof}
	Since $J \LRp{\boldsymbol{u}, p, m} \ge 0$ for any admissible $\LRp{\boldsymbol{u}, p, m}$ with $m \in \mc{M}_{ad}$, there exists a sequence $\{ m_n \}_{n=1}^{\infty}$, $m_n \in \mc{M}_{ad}$ such that the sequence of admissible triples $\{ \LRp{\boldsymbol{u}_n, p_n, m_n} \}_{n=1}^{\infty}$ satisfies
	\algn{
		\nonumber \lim_{n \to \infty} J \LRp{\boldsymbol{u}_n, p_n, m_n} =\inf_{\LRp{\boldsymbol{u}, p, m} \text{: admissible, } m \in \mc{M}_{ad}} J \LRp{\boldsymbol{u}, p, m} =: J_{\text{inf}} \ge 0 .
	}
	%
	Since $\{ J \LRp{\boldsymbol{u}_n, p_n, m_n} \}$ is bounded, $\{m_n\}$ is a bounded sequence in $L_t^2 L^2(\Omega)$ because $\|m_n\|_{L_t^2 L^2(\Omega)}^2 \le \frac{2}{\gamma} J \LRp{\boldsymbol{u}_n, p_n, m_n}$. By the energy estimate,
	\algn{
		\nonumber &\|\boldsymbol{u}_n\|_{L_t^2 \boldsymbol{V}} + \|\dot{\boldsymbol{u}}_n\|_{L_t^2 \boldsymbol{V}} + \|p_n\|_{L_t^2 L^2(\Omega)} + \|\dot{p}_n\|_{L_t^2 L^2(\Omega)} + \|p_n\|_{L_t^2 Q} \\
		\nonumber &\le C \LRp{\|\boldsymbol{u}_0\|_{\boldsymbol{V}} + \|p_0\|_{Q} + \|\boldsymbol{f}\|_{L_t^2 L^2(\Omega)} + \|m\|_{L_t^2 L^2(\Omega)}},
	}
	therefore $\{\boldsymbol{u}_n\}$, $\{\dot{\boldsymbol{u}}_n\}$ are bounded in $L_t^2 \boldsymbol{V}$, $\{p_n\}$, $\{\dot{p}_n\}$ are bounded in $L_t^2 L^2(\Omega)$, and $\{p_n\}$ is bounded in $L_t^2 Q$. By Banach-Alaoglu theorem, there are weak limits $\boldsymbol{w}_0$, $\boldsymbol{w}_1$, $r_0$, $r_1$, $\tilde{r}_0$ such that
	\begin{align*}
		\boldsymbol{u}_n \rightharpoonup \boldsymbol{w}_0 \in L_t^2 \boldsymbol{V}, \quad \dot{\boldsymbol{u}}_n \rightharpoonup \boldsymbol{w}_1 \in L_t^2 \boldsymbol{V}, \quad p_n \rightharpoonup r_0 \in L_t^2Q, \quad \dot{p}_n \rightharpoonup r_1 \in L_t^2 L^2(\Omega).
	\end{align*}
	Since $\LRp{\boldsymbol{u}_n, p_n, m_n}$ is admissible,
	\algn{ \label{eq:admissible-seq}
		\int_0^T \LRp{a_{\boldsymbol{u}} \LRp{\dot{\boldsymbol{u}}_n (t), \boldsymbol{v} (t)} + b \LRp{\dot{p}_n (t), \boldsymbol{v} (t)}} \, dt &= \int_0^T \LRp{\boldsymbol{f} (t), \boldsymbol{v} (t)}_{\Omega} \, dt, \\
		\nonumber \int_0^T \LRp{b \LRp{q(t), \dot{\boldsymbol{u}}_n (t)} - \LRp{s_0 \dot{p}_n (t), q(t)}_{\Omega} - a_p\LRp{ p_n (t), q(t)}} \, dt &= \int_0^T \LRp{m_n (t), q(t)}_{\Omega} \, dt
	}
	for all $\LRp{\boldsymbol{v}, q} \in C^{\infty} \LRp{[0, T]; \boldsymbol{V} \times Q}$ such that $\boldsymbol{v} (T) = 0$, $q(T) = 0$. If we take $n \to \infty$, then
	\algn{
		\nonumber \int_0^T \LRp{B_0 \LRp{\LRp{\boldsymbol{w}_1 (t), r_1 (t)}, \LRp{\boldsymbol{v} (t), q(t)}} - a_p\LRp{ \tilde{r}_0 (t), q(t)}} \, dt \\
		\nonumber = \int_0^T \LRp{\LRp{\boldsymbol{f}(t), \boldsymbol{v} (t)}_{\Omega} + \LRp{\bar{m}(t), q(t)}_{\Omega}} \, dt.
	}
	By the integration by parts in time, \eqref{eq:admissible-seq} can be rewritten as
	\algn{
		\nonumber &\int_0^T \LRp{-B_0 \LRp{\LRp{\boldsymbol{u}_n (t), p_n (t)}, \LRp{\dot{\boldsymbol{v}}(t), \dot{q}(t)}} - a_p\LRp{ p_n (t), q (t)}} \, dt \\ 
		\notag
		& \ \ + a_{\boldsymbol{u}} \LRp{\boldsymbol{u}_0, \boldsymbol{v} (0)} + b \LRp{p_0, \boldsymbol{v} (0)} + b \LRp{q(0), \boldsymbol{u}_0} - \LRp{s_0 p_0, q (0)}_{\Omega} \\ \notag
		&= \int_0^T \LRp{\LRp{\boldsymbol{f} (t), \boldsymbol{v} (t)}_{\Omega} + \LRp{m_n (t), q (t)}_{\Omega}} \, dt,
	}
	where we use $\boldsymbol{u}_n (0) = \boldsymbol{u}_0$, $p_n (0) = p_0$. If $n \to \infty$, the above converges to
	\algn{ \label{eq:admissible-convergence}
		&\int_0^T \LRp{-B_0 \LRp{\LRp{\boldsymbol{w}_0 (t), r_0 (t)}, \LRp{\dot{\boldsymbol{v}}(t), \dot{q}(t)}} - a_p\LRp{ \tilde{r}_0 (t), q (t)}} \, dt \\ \notag
		& \ \ + a_{\boldsymbol{u}} \LRp{\boldsymbol{u}_0, \boldsymbol{v} (0)} + b \LRp{p_0, \boldsymbol{v} (0)} + b \LRp{q(0), \boldsymbol{u}_0} - \LRp{s_0 p_0, q (0)}_{\Omega} \\ \notag
		&= \int_0^T \LRp{\LRp{\boldsymbol{f} (t), \boldsymbol{v} (t)}_{\Omega} + \LRp{\bar{m} (t), q (t)}_{\Omega}} \, dt .
	}
    Since $\LRp{\boldsymbol{w}_0, r_0, \bar{m}}$ satisfies \eqref{eq:admissible-convergence}, $\LRp{\boldsymbol{w}_0, r_0, \bar{m}}$ is admissible. Moreover, the Sobolev embedding $H^1(0,T; X) \hookrightarrow C^0([0,T]; X)$ means that the evaluation operator at $t=T$ is bounded from $H^1(0,T; \boldsymbol{V})\times H^1(0,T;L^2(\Omega))$ to $\boldsymbol{V} \times L^2(\Omega)$, so $J ({\boldsymbol{w}}_0, r_0, \bar{m})$ and $J ({\boldsymbol{u}}_n, p_n, \bar{m}_n)$ are well-defined. In particular, $J$ is weakly lower semi-continuous, so 
	\algn{
		\nonumber J \LRp{{\boldsymbol{w}}_0, r_0, \bar{m}} \le \lim_{n \to \infty} \inf J \LRp{\boldsymbol{u}_n, p_n, m_n} = J_{\inf}.
	}
	For uniqueness, suppose that there are two solutions $\LRp{\bar{\boldsymbol{u}}_0, \bar{p}_0, \bar{m}_0}$ and $\LRp{\bar{\boldsymbol{u}}_1, \bar{p}_1, \bar{m}_1}$ such that
	$J \LRp{\bar{\boldsymbol{u}}_0, \bar{p}_0, \bar{m}_0} = J \LRp{\bar{\boldsymbol{u}}_1, \bar{p}_1, \bar{m}_1} = J_{\inf}$. 
	$\mathcal{M}_{ad}$ is convex, so $\bar{m} := \frac{1}{2} \LRp{\bar{m}_0 + \bar{m}_1} \in \mathcal{M}_{ad}$ with the admissible triple $(\bar{\boldsymbol{u}}, \bar{p}, \bar{m})$. Since the map $m \mapsto \LRp{\boldsymbol{u}, p, m}$ from $\mathcal{M}_{ad}$ to the set of admissible triple $\LRp{\boldsymbol{u}, p, m}$ is affine, and $J$ is strictly convex, so
	\algn{
		\nonumber J \LRp{\bar{\boldsymbol{u}}, \bar{p}, \bar{m}} < J \LRp{\bar{\boldsymbol{u}}_0, \bar{p}_0, \bar{m}_0}. 
	}
	This contradicts the assumption that $\bar{m}_0$ is a solution of the optimal control problem, so a solution of the optimal control problem is unique.
\end{proof}
\subsection{Necessary optimality conditions}
To derive the first order necessary optimality conditions, we use the adjoint operator of control-to-state operator $\mathcal{S}$. 
\begin{definition}
	Let $\mathcal{S}:\mathcal{M}_{ad} \rightarrow \boldsymbol{V} \times Q$ be the control-to-state operator that maps $m$ to $\boldsymbol{y}=(\boldsymbol{u},p)$, where $(\boldsymbol{u},p)$ is the weak solution to the continuous problem with the initial condition, boundary and interior sources all set to zero except the control $m$.
\end{definition}
Since $L^{2}(0,T;\boldsymbol{V} \times Q)$ is continuously embedded in $\boldsymbol{Y} := L^{2}(0,T;\boldsymbol{L}^2(\Omega) \times L^2(\Omega))$. Then $\mathcal{S}:\mathcal{M}_{ad} \rightarrow \boldsymbol{Y}$ is still linear and continuous. We can write the optimal control problem equivalently as to find $\underset{m \in \mathcal{M}_{ad}}{\min} J_r(m),$
where the reduced cost functional $J_r: \mathcal{M}_{ad} \rightarrow \mathbb{R}$ is defined as
\begin{align*}
	J_r(m) &=  \frac{\gamma}{2} \| m \|_{L^2(0,T; L^2(\Omega))}^2 + \frac{\alpha_{\boldsymbol{u}, C}}{2} \int_0^T \int_{\Omega} |\mathcal{S}_{\boldsymbol{u}}(m) \LRp{x, t} + \boldsymbol{u}_{0m} - \boldsymbol{u}_C \LRp{x, t}|^2 \, dx \, dt  \\
	&\quad + \frac{\alpha_{p, C}}{2} \int_0^T \int_{\Omega} |\mathcal{S}_p(m)\LRp{x, t} + p_{0m} - p_C (x, t)|^2 \, dx \, dt,
\end{align*}
where $\mathcal{S}_{\boldsymbol{u}}$ and $\mathcal{S}_{p}$ just denote the two components obtained from the control-to-state operator $\mathcal{S}$ corresponding to control $m$ and $(\boldsymbol{u}_{0m},p_{0m})$ denote a weak solution to the state system with the control $m=0$ and the initial condition, boundary and interior sources are set as desired. 
Since $\mathcal{S}:\mathcal{M}_{ad} \rightarrow \boldsymbol{Y}$ is linear, 
the cost functional $J_r:\mathcal{M}_{ad} \rightarrow \mathbb{R}$ is Frech\'et differentiable. 

The corresponding variational equations for the adjoint problem are given by
\begin{subequations}\label{eq:new-weakadj}
	\begin{align*}
		-a_{\boldsymbol{u}} (\dot{\boldsymbol{w}}, \boldsymbol{v}) - b(\dot{r},\boldsymbol{v}) &=  \alpha_{\boldsymbol{u}, C} ( \boldsymbol{u}-\boldsymbol{u}_C, \boldsymbol{v} )_{\Omega}, \\
		-b(q,\dot{\boldsymbol{w}}) + (s_0 \dot{r}, q)_{\Omega} - a_p (r, q) &= \alpha_{p, C}( p-p_C, q)_{\Omega},
	\end{align*}
\end{subequations}
for all $\boldsymbol{v} \in \boldsymbol{V}$ and $q\in Q$. Here, we use the notations $\boldsymbol{w}$ and $r$ to represent the corresponding displacement and pressure variables, respectively. The space-time integration by parts to derive the adjoint equation leads to the following terminal conditions ${\boldsymbol{w}}(\cdot,T) = \boldsymbol{0}, {r}(\cdot,T) = 0$. 
\begin{lemma}
	$\bar{m}$ is the optimal control solution to the OCP if and only if
	\begin{align}\label{ctsvi}
		\left(\gamma \bar{m} + r ,m-\bar{m}\right)_{(0,T) \times \Omega} \ge 0 \qquad \forall m \in \mathcal{M}_{ad}.
	\end{align}
\end{lemma}
\begin{proof}
	The result is a direct consequence of \cite[Theorem~3.7]{MR4410836}.
\end{proof}
	Using the projection $\mathcal{P}_{\mathcal{M}_{ad}}:L^{2}(0,T;L^{2}(\Omega)) \rightarrow {\mathcal{M}_{ad}}$ (see \cite[Theorem~3.3]{MR3022219}) defined as:
	\begin{align}\label{mmproj}
		\mathcal{P}_{\mathcal{M}_{ad}}(r)(t,x):= \max\{m_{a},\min\{r(t,x), m_{b}\}\}, \quad \text{for a.e.} \ (t,x) \in (0,T) \times \Omega,
	\end{align}
the continuous \textit{variational inequality} \eqref{ctsvi} can be expressed as:
		\begin{align}\label{projcts}
			\bar{m} = \mathcal{P}_{\mathcal{M}_{ad}} \left(-\frac{1}{\gamma} r(\bar{m})\right).
	\end{align}
\begin{remark}
	Since $r \in L_t^{\infty} H^1(\Omega) \cap L_t^{2} H^1(\Omega)$, the projection \eqref{projcts} is well-defined for almost every $t\in (0,T)$. 
\end{remark}
%
\section{Discrete formulation}\label{Discrete formulation}
In this section we discuss discretizations of the optimal control problem with finite elements. 

\subsection{Spatial and space-time Discretizations}
To construct the finite element approximation, we consider a family of shape-regular partition $\{\mathcal{T}_h\}$ of $\overline{\Omega}$ into triangles or rectangles $K$ with diameter $h_K$. 

Consider a pair of finite elements $\boldsymbol{V}_h \times Q_h \subset \boldsymbol{V}\times Q$ which satisfies the inf-sup condition: there exists $C>0$ independent of $h$ such that 
\algn{
	\label{eq:Vh-Qh-inf-sup}
	\inf_{0\not = q \in Q_h} \sup_{\boldsymbol{v}\in \boldsymbol{V}_h} \frac{( q,  \dive \boldsymbol{v})_{\Omega}}{\|\boldsymbol{v}\|_{H^1(\Omega)} \|q\|_{L^2(\Omega)}} \ge C .
}
It is known that the MINI-element \cite[Sections 8.6 and 8.7]{BMO} and 
%
the generalized Taylor-Hood-$\mathbb{P}_l$ finite elements satisfy \eqref{eq:Vh-Qh-inf-sup}. 
%

For time discretization let $\Delta t$ represent the time step size, where $T=N\Delta t$ and $N$ is the positive integer. We introduce the notation $t_k=k\Delta t$ and $I_k := (t_{k}, t_{k+1}]$ for $k=0,\cdots,N-1$. For a continuous function $g$ defined on $[0,T]$, let $g^k=g(t_k)$. For a given sequence $\{g^k\}_{k\ge 0}$, the derivative is then approximated as follows:
\[\partial_\ta g^{k+1}:=\frac{g^{k+1}-g^k}{\Delta t}.\]
Now, we set
\begin{align*}
	\boldsymbol{V}_{hk} &:= \{\boldsymbol{v} \in L^2(0,T; \boldsymbol{V}_h): \boldsymbol{v}|_{I_{k}} \in \mathbb{P}_{0}(I_k; \boldsymbol{V}_h) \ \text{for} \ k = 0,\cdots,N-1\},
    \\
	Q_{hk} &:= \{q \in L^2(0,T; Q_h) : q|_{I_{k}} \in \mathbb{P}_{0}(I_k; Q_h) \ \text{for} \ k = 0,\cdots,N-1\},
\end{align*}
which means that $\boldsymbol{v} \in \boldsymbol{V}_{hk}$ and $q \in Q_{hk}$ are piecewise constant polynomials with respect to time.
\subsection{Finite element approximation}\label{Finite element approximation:}
In the dG(0) method, on each interval $I_k = (t_k, t_{k+1}]$, the discrete solution is sought as a constant in time, i.e., $(U,P)|_{I_k} \equiv (U^{k+1}, P^{k+1}) \in \boldsymbol{V}_h \times Q_h$. The scheme is obtained by integrating the equations over $I_k$ and testing against constant-in-time test functions, so that
\[
  \frac{1}{\Delta t}\int_{I_k} \partial_{\tau} \boldsymbol{u}\,dt \approx \frac{U^{k+1}-U^k}{\Delta t} =: \partial_\tau U^{k+1},
\]
and analogously for $P$. Throughout this section the notation
\[
  \LRp{g, v}_{I_k \times \Omega} := \int_{t_k}^{t_{k+1}}\int_{\Omega} g\, v\,dx\,dt = \Delta t \int_\Omega \bar{g}^{k+1} v\,dx,
  \qquad \bar{g}^{k+1} := \frac{1}{\Delta t}\int_{t_k}^{t_{k+1}} g\,dt,
\]
is used consistently for any source function $g$.

\medskip
{\bf Discretization of forward equation} The numerical solution at the $(k+1)-$th time step is defined inductively as follows: For $0\le k \le N-1$ find $(U^{k+1}, P^{k+1}) \in \boldsymbol{V}_h \times Q_h$ such that
\begin{subequations}\label{eq:fullydiscrete-eqs}
	\algn{
		\label{eq:fullydiscrete-eq1}
		a_{\boldsymbol{u}} \LRp{\partial_\ta U^{k+1}, \boldsymbol{v}} + b \LRp{\partial_\ta P^{k+1}, \boldsymbol{v}} &= \frac{1}{\Delta t} \LRp{{\boldsymbol{f}}, \boldsymbol{v}}_{I_k\times\Omega}, \\
		\label{eq:fullydiscrete-eq2}
		b \LRp{q, \partial_\ta U^{k+1}} - \LRp{s_0 \partial_\ta P^{k+1}, q}_{\Omega} - a_p\LRp{ P^{k+1}, q} &= \frac{1}{\Delta t}\LRp{m, q}_{I_k\times\Omega}, 
	}    
\end{subequations}
for all $\LRp{\boldsymbol{v}, q} \in \boldsymbol{V}_h \times Q_h$. 
\begin{lemma}
    The system \eqref{eq:fullydiscrete-eqs} has a unique solution.     
\end{lemma}
\begin{proof}
To establish the well-posedness of \eqref{eq:fullydiscrete-eqs}, we show that $(U^{k+1}, P^{k+1})$ is uniquely determined by the linear system (\ref{eq:fullydiscrete-eqs})  when $U^{k}, P^{k}, {\boldsymbol{f}}$ and $m$ are given. Equivalently, it is enough to show that $U^{k+1}=0$,  $P^{k+1}=0$ when $U^{k}, P^{k}, \boldsymbol{f}, m$ are zero. If we assume that $U^{k}, P^{k}, \boldsymbol{f}, m$ are zero, \eqref{eq:fullydiscrete-eqs} can be written as:
\begin{subequations}
	\algns{
		a_{\boldsymbol{u}} \LRp{U^{k+1}, \boldsymbol{v}} + b \LRp{P^{k+1}, \boldsymbol{v}} = 0,
		\\
		b \LRp{q, U^{k+1}} - \LRp{s_0 P^{k+1}, q}_{\Omega} - \Delta t\LRp{{\kappa} \nabla P^{k+1}, \nabla q}_{\Omega} = 0, 
	}    
\end{subequations}
for all $\LRp{\boldsymbol{v}, q} \in \boldsymbol{V}_h \times Q_h$. Choosing $ \boldsymbol{v}=U^{k+1}$ and $q=-P^{k+1}$ and adding both equations, we have
\begin{align*}
	0\le  a_{\boldsymbol{u}} \LRp{U^{k+1}, U^{k+1}}+  \LRp{s_0 P^{k+1}, P^{k+1}}_{\Omega} +\Delta t a_p\LRp{ P^{k+1}, P^{k+1}}=0.
\end{align*}
This implies that $U^{k+1}=0 = P^{k+1}$.    
\end{proof}


{\bf Discrete optimal control problem} {Consider the discrete optimal control problem: Find $m \in \mathcal{M}_{ad}$ such that 
	\begin{align}
	\label{discr_obj} 
    \min_{m \in \mathcal{M}_{ad}} J_h \LRp{U, P, m} := \frac{\alpha_{\boldsymbol{u}, C} }{2} \sum_{k=0}^{N-1} \|U^{k+1} - \boldsymbol{u}_C\|_{L^{2}(I_k; L^2(\Omega))}^2
    \\ \notag
		  +\hspace{-0.45mm} \frac{\alpha_{p, C} }{2}\sum_{k=0}^{N-1} \|P^{k+1} - p_C \|_{L^{2}(I_k; L^2(\Omega))}^2
          +\hspace{-0.45mm} \frac{\gamma}{2} \|m\|_{L^{2}(0,T; L^2(\Omega))}^2,
	\end{align}
    with $(U, P) \in \boldsymbol{V}_{hk} \times Q_{hk}$ 
	subject to the fully discrete state scheme \eqref{eq:fullydiscrete-eqs}.
    We use a variational discretization approach introduced in \cite{MR2122182} for the controls, i.e., we do not fix a finite dimensional approximation of the control space. 
\begin{theorem}
	There exists a unique solution in $\boldsymbol{V}_{hk} \times Q_{hk} \times \mathcal{M}_{ad}$ to the discrete optimal control problem \eqref{eq:fullydiscrete-eqs}-\eqref{discr_obj}.
\end{theorem}
\begin{proof}
	The proof follows along the same lines as the continuous case discussed in Section~\ref{Existence of optimal control} with finite-dimensional state compactness.
\end{proof}
\begin{theorem}
	(Discrete first order optimality system). A control $m \in \mc{M}_{ad}$ and the corresponding associated state $(U, P) \in \boldsymbol{V}_{hk} \times Q_{hk}$ is the optimal solution to the problems \eqref{discr_obj}, \eqref{eq:fullydiscrete-eqs} if and only if there exists an adjoint state $(W, R) \in \boldsymbol{V}_{hk} \times Q_{hk}$ which satisfies the following discrete adjoint equation:
	\begin{subequations}
		\label{eq:fullydiscreteadj-eqs11}
		\begin{align}
			\label{eq:fullydiscreteadj-eq12}
			-a_{\boldsymbol{u}} \LRp{\partial_\ta W^{k+1}, \boldsymbol{v}} - b \LRp{\partial_\ta R^{k+1}, \boldsymbol{v}} = \frac{1}{\Delta t} \alpha_{\boldsymbol{u}, C} (U^{k+1} - \boldsymbol{u}_{C}, \boldsymbol{v})_{I_k \times \Omega},  
			\\
			\label{eq:fullydiscreteadj-eq23}
			-b \LRp{q, \partial_\ta {W}^{k+1}} + \LRp{s_0 \partial_\ta {R}^{k+1}, q}_{\Omega} - a_p\LRp{ {R}^{k}, q} = \frac{1}{\Delta t} \alpha_{p, C} ({P}^{k+1} - p_C, q)_{I_k \times \Omega},
		\end{align}
	\end{subequations}
	for all $\LRp{\boldsymbol{v}, q} \in \boldsymbol{V}_h \times Q_h$, $k=0,1,\cdots,N-1$, with the terminal conditions $W^N = 0, \quad R^N = 0$ and the discrete \textit{variational inequality} written as 
		\begin{align}
			\label{eq:varineq}
			\sum_{k=0}^{N-1}(\gamma {m}+ R^{k},z - {m})_{I_k\times \Omega} \ge 0 \qquad \forall z \in \mc{M}_{ad}. 
		\end{align}
\end{theorem}
\begin{proof}
	Similarly to the proof in the continuous counterpart, we can obtain that the optimality of $m$ is equivalent to the existence of a discrete adjoint state $({W}^{k}, {R}^{k})$ which satisfies \eqref{eq:fullydiscreteadj-eqs11} and the discrete \textit{variational inequality}. The $R^k$ in $a_p(\cdot, \cdot)$ follows from transposing dG(0) scheme in the pressure equation. 
\end{proof}
The adjoint state system \eqref{eq:fullydiscreteadj-eqs11} can be written as:
\begin{subequations}
	\label{eq:fullydiscreteadj-eqs11a}
	\begin{align}
		\label{eq:fullydiscreteadj-eq12a}
		  a_{\boldsymbol{u}} \LRp{W^{k}, \boldsymbol{v}} + b \LRp{R^{k}, \boldsymbol{v}} &=  \alpha_{\boldsymbol{u}, C}  (U^{k+1} - \boldsymbol{u}_{C}, \boldsymbol{v})_{I_k \times \Omega} \\
		  \notag &\quad \ + a_{\boldsymbol{u}} \LRp{W^{k+1}, \boldsymbol{v}} + b \LRp{R^{k+1}, \boldsymbol{v}},
		\\
		\label{eq:fullydiscreteadj-eq23a}
		b \LRp{q,W^{k}} - \LRp{s_0 R^{k}, q}_{\Omega} - \Delta t a_p\LRp{ R^{k}, q} &= \alpha_{p, C}  (P^{k+1} - p_C, q)_{I_k \times \Omega}  \\
		\notag &\quad \ + b \LRp{q, W^{k+1}} - \LRp{s_0 R^{k+1}, q}_{\Omega},
	\end{align}    
\end{subequations}
for $k=0,1,\cdots,N-1$.
 The \textit{variational inequality} \eqref{eq:varineq} combined with the state and adjoint systems \eqref{eq:fullydiscrete-eqs} and \eqref{eq:fullydiscreteadj-eqs11a}, respectively, form a fully discrete optimality system for the optimal control problem. The discrete control variable is obtained by employing a discrete version of the projection formula \eqref{mmproj} as (see \cite[Section~4]{MR2644299} for more details):
 \begin{align}\label{proj}
 	{m |_{I_k \times \Omega} = \mathcal{P}_{\mathcal{M}_{ad}} \left(-\frac{1}{\gamma}R^{k}\right) \qquad \text{for} \ k=0, 1,\cdots,N-1.}
 \end{align} 
 An important consequence from this formula is that the solution $m \in \mathcal{M}_{ad}$ of the discrete optimal control problem \eqref{discr_obj} is constant on $I_k$, so $m^k$ is well-defined.
 
 This discrete optimality system is solved using a time dependent version of standard fixed point algorithm described in Sec.~\ref{Numerical Experiments}.}

\section{A priori error analysis}\label{A priori Error analysis}
In this section, we derive the a priori error estimates of the fully discrete scheme proposed in section~\ref{Discrete formulation}. Throughout this section, the right-hand side $\boldsymbol{f}$, the initial data $(\boldsymbol{u}(0), p(0)) \in \boldsymbol{V} \times Q$, the observation data $\boldsymbol{u}_C$, $p_C$, and the coefficients $\alpha_{p,C}$, $\alpha_{u,C}$, $\gamma$ are fixed, and all optimal control problems use the same initial/observation data and the coefficients. Moreover, all discrete forward equations use the same numerical initial data $(U^0, P^0) \in \boldsymbol{V}_h \times Q_h$ which satisfying 
\begin{align}
    \label{eq:initial-data-assumption}
    \|\boldsymbol{u}(0)-U^0\|_{H^1(\Omega)}+\|p(0)-P^0\|_{H^1(\Omega)}\le C h^{s}(\|\boldsymbol{u}(0)\|_{H^{1+s}(\Omega)}+\|p(0)\|_{H^{1+s}(\Omega)}), \quad s >0.
\end{align}
Note that $(U^0, P^0)$ do not need to satisfy the equation \eqref{sys1} for $\tilde{\boldsymbol{f}}(0)$. 

\subsection{Analysis of auxiliary forward equation}
Let us denote $\Pi_h^{\boldsymbol{u}} : \boldsymbol{V} \to \boldsymbol{V}_h$ and $\Pi_h^{p}: Q \to Q_h$ the elliptic projections defined by 
\begin{align}
    \label{eq:elasticity-elliptic-projection}
    a_{\boldsymbol{u}} (\Pi_h^{\boldsymbol{u}} \boldsymbol{v}, \boldsymbol{w}) &=     a_{\boldsymbol{u}} ( \boldsymbol{v}, \boldsymbol{w}) \qquad \forall \boldsymbol{w} \in \boldsymbol{V}_h,
    \\
    \label{eq:poisson-elliptic-projection}
    a_p( \Pi_h^p q, r) &= a_p( q, r)_{\Omega} \qquad \forall r \in Q_h . 
\end{align}
It is well-known that 
\begin{align}
    \label{eq:elliptic-projection-approx}   
    \| \boldsymbol{v} - \Pi_h^{\boldsymbol{u}} \boldsymbol{v} \|_{a_{\boldsymbol{u}}} \le Ch^{s} \|\boldsymbol{v}\|_{H^{1+s}(\Omega)}, 
    \qquad 
    \| q - \Pi_h^{p} q \|_{a_{p}} \le Ch^{s} \|q\|_{H^{1+s}(\Omega)} 
\end{align}
for $s \ge 0$. 
%
\begin{lemma}\label{lemma:aux-forward-error}
	Let $({\boldsymbol{u}},{p},{m})$ be the solution of the optimal control problem \eqref{obj_fun}. Suppose that $(U({m}), P({m})) \in \boldsymbol{V}_{hk} \times Q_{hk}$ is a solution of \eqref{eq:fullydiscrete-eqs} for the continuous optimal control ${m}$ with numerical initial data $(U^0, P^0)$ satisfying \eqref{eq:initial-data-assumption} for some $0<s\le 1$. Assume also that 
    \[
        \boldsymbol{u}\in H^1(0,T;\boldsymbol{H}^{1+s}(\Omega)), \quad  p\in H^1(0,T;H^{1+s}(\Omega)).
    \]
    Then, 
	\begin{align}
        \label{eq:forward-aux-estimate}
		&\max_{1\le k \le N} \LRp{\| \boldsymbol{u}^k - U^k(m)\|_{H^1(\Omega)} + \| p^k - P^k(m)\|_{L^2(\Omega)} } 
        \\
        \notag
		&+ \| \boldsymbol{u} - U(m)\|_{L^2(0,T; H^1(\Omega))} + \| p - P(m)\|_{L^2(0,T; L^2(\Omega))} 
        \\
        \notag
        &\le h^{s}(\|\boldsymbol{u}(0)\|_{H^{1+s}(\Omega)}+\|p(0)\|_{H^{1+s}(\Omega)}) 
        \\
        \notag
        &\quad + C(h^s+\Delta t)
        \left(
        \|\boldsymbol{u}\|_{H^1(0,T;H^{1+s}(\Omega))}
        +
        \|p\|_{H^1(0,T;H^{1+s}(\Omega))}
        \right) .
    \end{align}
\end{lemma}
\begin{proof}
    We start our error analysis to split the error $({\boldsymbol{u}}^k-U^k({m}), {p}^k-P^k({m}))$ into two parts as follows: 
\begin{align*}
	{\boldsymbol{u}}^j-U^j({m})&={\boldsymbol{u}}^j-\Pi_h^{\boldsymbol{u}} {\boldsymbol{u}}^j+\Pi_h^{\boldsymbol{u}} {\boldsymbol{u}}^j-U^j({m}) =: e_{I,j}^{\boldsymbol{u}}+\theta^j_{\boldsymbol{u}}; &&\hspace{-1mm} j = 0,1,\cdots,N,
    \\
	{p}^j-P^j(m)&={p}^j-\Pi_h^p {p}^j+\Pi_h^p {p}^j-P^j({m}) =: e_{I,j}^p+\theta^j_p;&&\hspace{-1mm} j = 0,1,\cdots,N.
\end{align*}
Regarding
\begin{align*}
    (\boldsymbol{f}, \boldsymbol{v})_{I_k \times \Omega} &=\int_{I_k} \LRs{a_{\boldsymbol{u}}(\dot{{\boldsymbol{u}}}(s), \boldsymbol{v}) + b(\dot{{p}}(s), \boldsymbol{v})} \,ds 
    \\
    &= a_{\boldsymbol{u}}({\boldsymbol{u}}^{k+1} - {\boldsymbol{u}}^k, \boldsymbol{v}) + b({p}^{k+1} - {p}^k, \boldsymbol{v}),
    \\
    ({m}, q)_{I_k \times \Omega} &=\int_{I_k} \LRs{ b(q, \dot{{\boldsymbol{u}}}) - (s_0 \dot{{p}}(s), q)_{\Omega} - a_p( {p}(s), q) } \,ds 
    \\
    &= b(q, {\boldsymbol{u}}^{k+1} - {\boldsymbol{u}}^k) - (s_0 ({p}^{k+1} - {p}^k), q)_{\Omega} - \int_{I_k} a_p(p, q) \,ds 
\end{align*}
for $\boldsymbol{v} \in \boldsymbol{V}_h$, $q \in Q_h$, and the elliptic projections \eqref{eq:elasticity-elliptic-projection}, \eqref{eq:poisson-elliptic-projection}, we obtain the following error equations:
for all $(\boldsymbol v,q)\in \boldsymbol V_h\times Q_h$,
\begin{align}
\label{eq:theta-error-1}
a_{\boldsymbol u}(\partial_\tau\theta_{\boldsymbol u}^k,\boldsymbol v)
+
b(\partial_\tau\theta_p^k,\boldsymbol v)
&=
-b(\eta_p^k,\boldsymbol v),
\\
\label{eq:theta-error-2}
b(q,\partial_\tau\theta_{\boldsymbol u}^k)
-
(s_0\partial_\tau\theta_p^k,q)_\Omega
-
a_p(\theta_p^k,q)
&=
-b(q,\boldsymbol\eta_{\boldsymbol u}^k)
+
(s_0\eta_p^k,q)_\Omega
+
D^k(q).
\end{align}
%
%
%
where 
\begin{align*}
        \eta_p^k
        &:=
        \partial_\tau p^k-\partial_\tau\Pi_h^p p^k
        =
        (I-\Pi_h^p)\partial_\tau p^k,
        \\
        \boldsymbol\eta_{\boldsymbol u}^k
        &:=
        \partial_\tau\boldsymbol u^k
        -
        \partial_\tau\Pi_h^{\boldsymbol u}\boldsymbol u^k
        =
        (I-\Pi_h^{\boldsymbol u})\partial_\tau\boldsymbol u^k,
        \\
        D^k(q)
        &:=
        \frac1{\Delta t}
        \int_{t_{k-1}}^{t_k}
        a_p(p(s)-p^k,q)\,ds.
\end{align*}

We now derive a stability estimate directly from
\eqref{eq:theta-error-1}--\eqref{eq:theta-error-2}. Introduce
\[
        F^k(\boldsymbol v):=-b(\eta_p^k,\boldsymbol v), \qquad G^k(q)
        :=
        -b(q,\boldsymbol\eta_{\boldsymbol u}^k)
        +
        (s_0\eta_p^k,q)_\Omega
        +
        D^k(q).
\]
Let $A_h:\boldsymbol V_h\to \boldsymbol V_h'$ and
$B_h:Q_h\to \boldsymbol V_h'$ be defined by
\[
        \langle A_h\boldsymbol w,\boldsymbol v\rangle
        :=
        a_{\boldsymbol u}(\boldsymbol w,\boldsymbol v),
        \qquad
        \langle B_h r,\boldsymbol v\rangle
        :=
        b(r,\boldsymbol v).
\]
From \eqref{eq:theta-error-1}, after summing in time, we obtain
\[
        A_h\theta_{\boldsymbol u}^k+B_h\theta_p^k
        = A_h\theta_{\boldsymbol u}^0+B_h\theta_p^0 + \Delta t\sum_{i=1}^k F^i =:
        H^k.
\]
From this we obtain
\[
        \|H^k\|_{\boldsymbol V_h'}
        \le C \LRp{ \|\theta_{\boldsymbol{u}}^0\|_{a_{\boldsymbol{u}}} + \|\theta_p^0\|_{L^2(\Omega)} + 
        \Delta t\sum_{i=1}^k \|F^i\|_{\boldsymbol V_h'}}, 
\]
and therefore,
\begin{equation}
\label{eq:H-bound-direct}
    \max_{1\le k\le N}\|H^k\|_{\boldsymbol V_h'}^2
    +
    \Delta t\sum_{k=1}^N\|H^k\|_{\boldsymbol V_h'}^2
    \le
    C\left(
    \|\theta_{\boldsymbol u}^0\|_{a_{\boldsymbol u}}^2
    +
    \|\theta_p^0\|_{L^2(\Omega)}^2
    +
    \Delta t\sum_{k=1}^N\|F^k\|_{\boldsymbol V_h'}^2
    \right). 
\end{equation}

By the inf-sup condition \eqref{eq:Vh-Qh-inf-sup},
\begin{equation}
\label{eq:theta-p-by-theta-u-H}
        \|\theta_p^k\|_{L^2(\Omega)}
        \le
        C\left(
        \|H^k\|_{\boldsymbol V_h'}
        +
        \|\theta_{\boldsymbol u}^k\|_{a_{\boldsymbol u}}
        \right).
\end{equation}
We next derive an energy estimate. Taking $q=\theta_p^k$ in
\eqref{eq:theta-error-2}, using
$B_h\theta_p^k=H^k-A_h\theta_{\boldsymbol u}^k$, and arguing as in the
standard energy method, we obtain
\[
\begin{aligned}
a_{\boldsymbol u}(\theta_{\boldsymbol u}^k,\partial_\tau\theta_{\boldsymbol u}^k)
+
(s_0\partial_\tau\theta_p^k,\theta_p^k)_\Omega
+
a_p(\theta_p^k,\theta_p^k)
=
\langle H^k,\partial_\tau\theta_{\boldsymbol u}^k\rangle
-
G^k(\theta_p^k).
\end{aligned}
\]
Multiplying by $\Delta t$ and summing from $k=1$ to $m\le N$ gives
\begin{align}
\label{eq:theta-energy-before-est}
&\frac12\|\theta_{\boldsymbol u}^m\|_{a_{\boldsymbol u}}^2
+
\frac{s_0}{2}\|\theta_p^m\|_{L^2(\Omega)}^2
+
\Delta t\sum_{k=1}^m
\|\theta_p^k\|_{a_p}^2
\\
&\qquad
\le \frac12\|\theta_{\boldsymbol u}^0\|_{a_{\boldsymbol u}}^2
+
\frac{s_0}{2}\|\theta_p^0\|_{L^2(\Omega)}^2 + 
\sum_{k=1}^m
\langle H^k,\theta_{\boldsymbol u}^k-\theta_{\boldsymbol u}^{k-1}\rangle
+
\Delta t\sum_{k=1}^m
|G^k(\theta_p^k)|.
\nonumber
\end{align}
The first term on the right-hand side is handled by summation by parts:
\[
\begin{aligned}
\sum_{k=1}^m
\langle H^k,\theta_{\boldsymbol u}^k-\theta_{\boldsymbol u}^{k-1}\rangle
&=
\langle H^m,\theta_{\boldsymbol u}^m\rangle
- 
\langle H^1,\theta_{\boldsymbol u}^0\rangle
-
\sum_{k=1}^{m-1}
\langle H^{k+1}-H^k,\theta_{\boldsymbol u}^k\rangle.
\end{aligned}
\]
Since
$H^{k+1}-H^k=\Delta t F^{k+1}$, Young's inequality yields, for any
$\varepsilon>0$,
\begin{align}
\label{eq:H-term-direct}
\sum_{k=1}^m
\langle H^k,\theta_{\boldsymbol u}^k-\theta_{\boldsymbol u}^{k-1}\rangle
&\le
\varepsilon\|\theta_{\boldsymbol u}^m\|_{a_{\boldsymbol u}}^2
+
C_\varepsilon\|H^m\|_{\boldsymbol V_h'}^2 + C \|H^1\|_{\boldsymbol{V}_h'}^2 + C \| \theta_{\boldsymbol{u}}^0\|_{a_{\boldsymbol{u}}}^2
\\
&\quad
+
\varepsilon\Delta t
\sum_{k=1}^{m-1}
\|\theta_{\boldsymbol u}^k\|_{a_{\boldsymbol u}}^2
+
C_\varepsilon\Delta t
\sum_{k=2}^{m}
\|F^k\|_{\boldsymbol V_h'}^2 .
\nonumber
\end{align}

We estimate $G^k(\theta_p^k)$ in the $a_p$-dual norm. Define
\[
        \|G^k\|_{a_p,*}
        :=
        \sup_{q\in Q_h\setminus\{0\}}
        \frac{|G^k(q)|}{\|q\|_{a_p}}.
\]
Then
\[
        |G^k(\theta_p^k)|
        \le
        \|G^k\|_{a_p,*}\|\theta_p^k\|_{a_p}.
\]
Therefore, by Young's inequality,
\begin{equation}
\label{eq:G-term-direct}
\Delta t\sum_{k=1}^m
|G^k(\theta_p^k)|
\le
\varepsilon
\Delta t\sum_{k=1}^m
\|\theta_p^k\|_{a_p}^2
+
C_\varepsilon
\Delta t\sum_{k=1}^m
\|G^k\|_{a_p,*}^2 .
\end{equation}
Combining \eqref{eq:theta-energy-before-est},
\eqref{eq:H-term-direct}, and \eqref{eq:G-term-direct}, choosing
$\varepsilon>0$ sufficiently small, and using \eqref{eq:H-bound-direct},
we get
\[
\begin{aligned}
\|\theta_{\boldsymbol u}^m\|_{a_{\boldsymbol u}}^2
+
s_0\|\theta_p^m\|_{L^2(\Omega)}^2
+
\Delta t\sum_{k=1}^m
\|\theta_p^k\|_{a_p}^2
&\le
C
\Delta t\sum_{k=1}^N
\left(
\|F^k\|_{\boldsymbol V_h'}^2
+
\|G^k\|_{a_p,*}^2
\right)
\\
&\quad
+
C\Delta t
\sum_{k=1}^{m-1}
\|\theta_{\boldsymbol u}^k\|_{a_{\boldsymbol u}}^2 + C(\|\theta_{\boldsymbol u}^0\|_{a_{\boldsymbol u}}^2
+
\|\theta_p^0\|_{L^2(\Omega)}^2) .
\end{aligned}
\]
By the discrete Gronwall inequality,
\begin{equation}
\label{eq:theta-u-ap-bound}
\max_{1\le m\le N}
\|\theta_{\boldsymbol u}^m\|_{a_{\boldsymbol u}}^2
+
\Delta t\sum_{k=1}^N
\|\theta_p^k\|_{a_p}^2
\le C \LRp{\|\theta_{\boldsymbol u}^0\|_{a_{\boldsymbol u}}^2
+
\|\theta_p^0\|_{L^2(\Omega)}^2} +
C
\Delta t\sum_{k=1}^N
\left(
\|F^k\|_{\boldsymbol V_h'}^2
+
\|G^k\|_{a_p,*}^2
\right).
\end{equation}
Using \eqref{eq:theta-p-by-theta-u-H},
\eqref{eq:H-bound-direct}, and \eqref{eq:theta-u-ap-bound}, we also obtain
\begin{equation}
\label{eq:theta-p-L2-bound}
\Delta t\sum_{k=1}^N
\|\theta_p^k\|_{L^2(\Omega)}^2
\le C \LRp{\|\theta_{\boldsymbol u}^0\|_{a_{\boldsymbol u}}^2
+
\|\theta_p^0\|_{L^2(\Omega)}^2} + 
C
\Delta t\sum_{k=1}^N
\left(
\|F^k\|_{\boldsymbol V_h'}^2
+
\|G^k\|_{a_p,*}^2
\right).
\end{equation}
Consequently,
\begin{align}
\label{eq:theta-final-stability}
&\max_{1\le m\le N}
\|\theta_{\boldsymbol u}^m\|_{a_{\boldsymbol u}} +
\left(
\Delta t\sum_{k=1}^N
\left(
\|\theta_{\boldsymbol u}^k\|_{a_{\boldsymbol u}}^2
+
\|\theta_p^k\|_{L^2(\Omega)}^2
\right)
\right)^{1/2}
\\
\notag
&\le C \LRp{\|\theta_{\boldsymbol u}^0\|_{a_{\boldsymbol u}} + \|\theta_p^0\|_{L^2(\Omega)}} + 
C
\left(
\Delta t\sum_{k=1}^N
\left(
\|F^k\|_{\boldsymbol V_h'}^2
+
\|G^k\|_{a_p,*}^2
\right)
\right)^{1/2}.
\end{align}

It remains to estimate $\|\theta_{\boldsymbol u}^0\|_{a_{\boldsymbol u}}$, $\|\theta_p^0\|_{L^2(\Omega)}$, $F^k$, and $G^k$. By \eqref{eq:initial-data-assumption}, the interpolation approximation \eqref{eq:elliptic-projection-approx}, Poincare inequality, and the triangle inequality, 
\begin{align}
    \|\theta_{\boldsymbol u}^0\|_{a_{\boldsymbol u}} + \|\theta_p^0\|_{L^2(\Omega)} \le Ch^{s} (\|\boldsymbol{u}(0)\|_{H^{1+s}(\Omega)} + \|p(0)\|_{H^{1+s}(\Omega)} ).
\end{align}
By the continuity of $b$ and Korn's inequality,
\[
        \|F^k\|_{\boldsymbol V_h'}
        \le
        C\|\eta_p^k\|_{L^2(\Omega)}.
\]
Moreover, from the definition of $G^k$ and Poincare's inequality,
\[
\begin{aligned}
\|G^k\|_{a_p,*}
&\le
C\|\boldsymbol\eta_{\boldsymbol u}^k\|_{H^1(\Omega)}
+
C s_0\|\eta_p^k\|_{L^2(\Omega)}
+
\|D^k\|_{a_p,*}.
\end{aligned}
\]
We now estimate the three residuals. Since
\[
        \partial_\tau \boldsymbol u^k
        =
        \frac1{\Delta t}
        \int_{t_{k-1}}^{t_k}
        \dot{\boldsymbol u}(s)\,ds,
\]
the approximation property of the elliptic projection gives
\[
\begin{aligned}
\Delta t\sum_{k=1}^N
\|\boldsymbol\eta_{\boldsymbol u}^k\|_{H^1(\Omega)}^2
&=
\Delta t\sum_{k=1}^N
\left\|
(I-\Pi_h^{\boldsymbol u})\partial_\tau\boldsymbol u^k
\right\|_{H^1(\Omega)}^2
\\
&\le
C h^{2s}
\Delta t\sum_{k=1}^N
\left(
\frac1{\Delta t}
\int_{t_{k-1}}^{t_k}
\|\dot{\boldsymbol u}(s)\|_{H^{1+s}(\Omega)}\,ds
\right)^2
\\
&\le
C h^{2s}
\|\dot{\boldsymbol u}\|_{L^2(0,T;H^{1+s}(\Omega))}^2 .
\end{aligned}
\]
Similarly,
\[
\begin{aligned}
\Delta t\sum_{k=1}^N
\|\eta_p^k\|_{L^2(\Omega)}^2
&=
\Delta t\sum_{k=1}^N
\left\|
(I-\Pi_h^p)\partial_\tau p^k
\right\|_{L^2(\Omega)}^2
\\
&\le
C h^{2(1+s)}
\|\dot p\|_{L^2(0,T;H^{1+s}(\Omega))}^2 .
\end{aligned}
\]
For the diffusion time residual, we use the $a_p$-dual norm. From the definition of $D^k(q)$ we get
\[
\begin{aligned}
\|D^k\|_{a_p,*}
&\le
\frac1{\Delta t}
\int_{t_{k-1}}^{t_k}
\|p(s)-p^k\|_{a_p}\,ds
\\
&\le
\frac1{\Delta t}
\int_{t_{k-1}}^{t_k}
\int_s^{t_k}
\|\dot p(r)\|_{a_p}\,dr\,ds
\\
&\le
C\Delta t^{1/2}
\|\dot p\|_{L^2(t_{k-1},t_k;H^1(\Omega))}.
\end{aligned}
\]
Therefore,
\[
\begin{aligned}
\Delta t\sum_{k=1}^N
\|D^k\|_{a_p,*}^2
&\le
C\Delta t^2
\|\dot p\|_{L^2(0,T;H^1(\Omega))}^2 .
\end{aligned}
\]

Combining the above bounds with \eqref{eq:theta-final-stability}, we obtain
\begin{align}
\label{eq:theta-est-final}
&\max_{1\le m\le N}
\|\theta_{\boldsymbol u}^m\|_{a_{\boldsymbol u}} +\left(
\Delta t\sum_{k=1}^N
\left(
\|\theta_{\boldsymbol u}^k\|_{a_{\boldsymbol u}}^2
+
\|\theta_p^k\|_{L^2(\Omega)}^2
\right)
\right)^{1/2} 
\\
\notag
&\le  Ch^{s} (\|\boldsymbol{u}(0)\|_{H^{1+s}(\Omega)} + \|p(0)\|_{H^{1+{s}}(\Omega)} )
\\
\nonumber
&\quad  +C\Big[
h^s
\|\dot{\boldsymbol u}\|_{L^2(0,T;H^{1+s}(\Omega))}
+
h^{1+s}
\|\dot p\|_{L^2(0,T;H^{1+s}(\Omega))}
+
\Delta t
\|\dot p\|_{L^2(0,T;H^1(\Omega))}
\Big].
\end{align}

To complete the proof, we estimate the difference of continuous and discrete solutions. Since 
\begin{align*}
    \|\boldsymbol{u}^k - U^k(m)\|_{a_{\boldsymbol{u}}} \le \|\boldsymbol{u}^k - \Pi_h^{\boldsymbol{u}} \boldsymbol{u}^k\|_{a_{\boldsymbol{u}}} + \|\theta_{\boldsymbol{u}}^k\|_{a_{\boldsymbol{u}}} 
    \le Ch^s \| u^k \|_{H^{1+s}(\Omega)} + \|\theta_{\boldsymbol{u}}^k\|_{a_{\boldsymbol{u}}},
\end{align*}
the $\max_{1\le k\le N}\| \boldsymbol{u}^k - U^k(m)\|_{a_{\boldsymbol{u}}}$ part in \eqref{eq:forward-aux-estimate} is proved. The $\max_{1\le k\le N}\| p^k - P^k(m)\|_{L^2(\Omega)}$ part in \eqref{eq:forward-aux-estimate} follows by \eqref{eq:theta-p-by-theta-u-H},  \eqref{eq:elliptic-projection-approx}, and 
\begin{align*}
        \|p^k - P^k(m)\|_{L^2(\Omega)} \le \| p^k - \Pi_h^{p} p^k\|_{L^2(\Omega)} + \|\theta_{p}^k\|_{L^2(\Omega)} 
    \le Ch^s \| p^k \|_{H^{1+s}(\Omega)} + \|\theta_{p}^k\|_{L^2(\Omega)}.
\end{align*}
For the space-time $L^2$ errors in \eqref{eq:forward-aux-estimate} define piecewise constant-in-time interpolations
\[
        \boldsymbol u_\tau(t)=\boldsymbol u^k,
        \qquad
        p_\tau(t)=p^k,
        \qquad t\in I_k.
\]
Then
\[
\begin{aligned}
\|\boldsymbol u-U(m)\|_{L^2(0,T;H^1(\Omega))}
&\le
\|\boldsymbol u-\boldsymbol u_\tau\|_{L^2(0,T;H^1(\Omega))}
\\
&\quad
+
\|\boldsymbol u_\tau-\Pi_h^{\boldsymbol u}\boldsymbol u_\tau\|_{L^2(0,T;H^1(\Omega))}
+
\|\theta_{\boldsymbol u}\|_{\ell_\tau^2(H^1(\Omega))}.
\end{aligned}
\]
Here we have 
\begin{align*}
        \|\boldsymbol u-\boldsymbol u_\tau\|_{L^2(0,T;H^1(\Omega))}
        &\le
        C\Delta t
        \|\dot{\boldsymbol u}\|_{L^2(0,T;H^1(\Omega))},
        \\
        \|\boldsymbol u_\tau-\Pi_h^{\boldsymbol u}\boldsymbol u_\tau\|_{L^2(0,T;H^1(\Omega))}
        &\le
        C h^{1+s}
        \|\boldsymbol u\|_{\ell_\tau^2(H^{1+s}(\Omega))}.
\end{align*}
Since $\|\theta_{\boldsymbol{u}}\|_{a_{\boldsymbol{u}}} \sim \|\theta_{\boldsymbol{u}}\|_{H^1(\Omega)}$, using \eqref{eq:theta-est-final},
\[
\begin{aligned}
&\|\boldsymbol u-U(m)\|_{L^2(0,T;H^1(\Omega))}
\le  Ch^{s} (\|\boldsymbol{u}(0)\|_{H^{1+s}(\Omega)} + \|p(0)\|_{H^{1+{s}}(\Omega)} )
\\
&+C\Big[
\Delta t
\|\dot{\boldsymbol u}\|_{L^2(0,T;H^1(\Omega))}
+
h^{1+s}
\|\boldsymbol u\|_{\ell_\tau^2(H^{1+s}(\Omega))}
\\
&\qquad +
h^s
\|\dot{\boldsymbol u}\|_{L^2(0,T;H^{1+s}(\Omega))}
+
h^{1+s}
\|\dot p\|_{L^2(0,T;H^{1+s}(\Omega))}
+
\Delta t
\|\dot p\|_{L^2(0,T;H^1(\Omega))}
\Big].
\end{aligned}
\]
Similarly,
\[
\begin{aligned}
\|p-P(m)\|_{L^2(0,T;L^2(\Omega))}
&\le
\|p-p_\tau\|_{L^2(0,T;L^2(\Omega))}
\\
&\quad
+
\|p_\tau-\Pi_h^p p_\tau\|_{L^2(0,T;L^2(\Omega))}
+
\|\theta_p\|_{\ell_\tau^2(L^2(\Omega))}.
\end{aligned}
\]
The first two terms satisfy
\[
        \|p-p_\tau\|_{L^2(0,T;L^2(\Omega))}
        \le
        C\Delta t
        \|\dot p\|_{L^2(0,T;L^2(\Omega))},
\]
and
\[
        \|p_\tau-\Pi_h^p p_\tau\|_{L^2(0,T;L^2(\Omega))}
        \le
        C h^{1+s}
        \|p\|_{\ell_\tau^2(H^{1+s}(\Omega))}.
\]
Using \eqref{eq:theta-est-final}, we obtain
\[
\begin{aligned}
&\|p-P(m)\|_{L^2(0,T;L^2(\Omega))}
\le  Ch^{s} (\|\boldsymbol{u}(0)\|_{H^{1+s}(\Omega)} + \|p(0)\|_{H^{1+{s}}(\Omega)} )
\\
&+C\Big[
\Delta t
\|\dot p\|_{L^2(0,T;L^2(\Omega))}
+
h^{1+s}
\|p\|_{\ell_\tau^2(H^{1+s}(\Omega))}
\\
&\qquad +
h^s
\|\dot{\boldsymbol u}\|_{L^2(0,T;H^{1+s}(\Omega))}
+
h^{1+s}
\|\dot p\|_{L^2(0,T;H^{1+s}(\Omega))}
+
\Delta t
\|\dot p\|_{L^2(0,T;H^1(\Omega))}
\Big].
\end{aligned}
\]

Combining the two estimates gives
\[
\begin{aligned}
&\|\boldsymbol u-U(m)\|_{L^2(0,T;H^1(\Omega))}
+
\|p-P(m)\|_{L^2(0,T;L^2(\Omega))}
\\
&\qquad
\le Ch^{s} (\|\boldsymbol{u}(0)\|_{H^{1+s}(\Omega)} + \|p(0)\|_{H^{1+{s}}(\Omega)} )
\\
&
+ C\Big[
\Delta t
\left(
\|\dot{\boldsymbol u}\|_{L^2(0,T;H^1(\Omega))}
+
\|\dot p\|_{L^2(0,T;L^2(\Omega))}
+
\|\dot p\|_{L^2(0,T;H^1(\Omega))}
\right)
\\
&\qquad
+
h^s
\|\dot{\boldsymbol u}\|_{L^2(0,T;H^{1+s}(\Omega))}
+
h^{s}
\left(
\|\boldsymbol u\|_{\ell_\tau^2(H^{1+s}(\Omega))}
+
\|p\|_{\ell_\tau^2(H^{1+s}(\Omega))}
+
\|\dot p\|_{L^2(0,T;H^{1+s}(\Omega))}
\right)
\Big].
\end{aligned}
\]
This completes the proof.
\end{proof}


\begin{lemma}[adjoint error for a fixed control]
\label{lem:adjoint-error-fixed-control}
Let $m\in \mathcal{M}_{ad}$ be fixed. Let $(\boldsymbol{u},p)$ be the
corresponding continuous state and let $(\boldsymbol{w},r)$ be the
corresponding continuous adjoint. Let $(U(m),P(m))$ be the fully
discrete state generated by the same control $m$, and let
$(W(m),R(m))$ be the corresponding fully discrete adjoint.

Assume that
\[
 r\in H^1(0,T;H^{1+s}(\Omega)),
 \quad 
 \boldsymbol{w}\in H^1(0,T;\boldsymbol{H}^{1+s}(\Omega)),
\]
for some $0<s\le 1$. Then
\begin{multline}
    \|R(m)-r\|_{L^2(0,T;L^2(\Omega))} 
    \\
    \le C\left( h^s+\Delta t + \|U(m)-\boldsymbol{u}\|_{L^2(0,T;L^2(\Omega))} +\|P(m)-p\|_{L^2(0,T;L^2(\Omega))} \right).
\end{multline}
Consequently, if the fixed-control state error satisfies
\[
\|U(m)-\boldsymbol{u}\|_{L^2(0,T;L^2(\Omega))}
+ \|P(m)-p\|_{L^2(0,T;L^2(\Omega))} \le C(h^s+\Delta t),
\]
then
\[
\|R(m)-r\|_{L^2(0,T;L^2(\Omega))} \le C(h^s+\Delta t).
\]
\end{lemma}
\begin{proof}
We introduce an auxiliary fully discrete adjoint
$(\widetilde W,\widetilde R)$ driven by the exact continuous state data. More precisely, the right-hand sides in the discrete adjoint equations are obtained from $\alpha_{\boldsymbol{u},C}(\boldsymbol{u}-\boldsymbol{u}_C)$ and $\alpha_{p,C}(p-p_C)$, instead of
$\alpha_{\boldsymbol{u},C}(U(m)-\boldsymbol{u}_C)$, $\alpha_{p,C}(P(m)-p_C)$. 
Then
\[
R(m)-r
=
\bigl(R(m)-\widetilde R\bigr)
+
\bigl(\widetilde R-r\bigr).
\]

We first estimate $R(m)-\widetilde R$. Subtracting the two discrete
adjoint systems gives a discrete adjoint problem with zero terminal data
and right-hand sides
\[
\alpha_{\boldsymbol{u},C}(U(m)-\boldsymbol{u}),
\qquad
\alpha_{p,C}(P(m)-p).
\]
By the stability estimate for the fully discrete adjoint problem, which is
the backward-in-time analogue of the state stability estimate, we obtain a system
\begin{align}
\label{eq:WR-error-1}
a_{\boldsymbol u}(\partial_\tau\theta_{W}^k,\boldsymbol v)
+
b(\partial_\tau\theta_R^k,\boldsymbol v)
&=
\frac{1}{\Delta t} \LRp{\alpha_{\boldsymbol{u},C}(U(m)-\boldsymbol{u}), \boldsymbol v}_{I_k \times \Omega},
\\
\label{eq:WR-error-2}
b(q,\partial_\tau\theta_{W}^k)
-
(s_0\partial_\tau\theta_R^k,q)_\Omega
-
a_p(\theta_R^k,q)
&= \frac{1}{\Delta t}
\LRp{\alpha_{p,C}(P(m)-p), q}_{I_k \times \Omega}
\end{align}
where 
\begin{align*}
    \theta_W^k := W(m)^k - \widetilde{W}^k, \quad \theta_R^k := R(m)^k - \widetilde{R}^k , \qquad 1 \le k \le N . 
\end{align*}
After reversing the time indices, this system has the same algebraic structure as the forward perturbation system. Therefore, the stability
argument used in Lemma~\ref{lemma:aux-forward-error}, particularly the procedures from \eqref{eq:theta-error-1}, \eqref{eq:theta-error-2} to \eqref{eq:theta-u-ap-bound}, yields  
\[
\|R(m)-\widetilde R\|_{L^2(0,T;L^2(\Omega))}
\le
C\left(
\|U(m)-\boldsymbol{u}\|_{L^2(0,T;L^2(\Omega))}
+
\|P(m)-p\|_{L^2(0,T;L^2(\Omega))}
\right)
\]
because $R(m)-\widetilde{R}$ is piecewise constant in time. 

It remains to estimate $\widetilde R-r$. Since
$(\widetilde W,\widetilde R)$ is the fully discrete approximation of
the continuous adjoint problem with the exact state data in the source
terms, the dG(0) error estimate as in Lemma~\ref{lemma:aux-forward-error} gives
\[
\|\widetilde R-r\|_{L^2(0,T;L^2(\Omega))}
\le
C(h^s+\Delta t),
\]
where the constant depends on the stated regularity of
$(\boldsymbol{w},r)$ and on the target data, but is independent of
$h$ and $\Delta t$.

Combining the two estimates and using the triangle inequality yields
\begin{multline}
\|R(m)-r\|_{L^2(0,T;L^2(\Omega))} 
\\
\le
C\left(
h^s+\Delta t
+
\|U(m)-\boldsymbol{u}\|_{L^2(0,T;L^2(\Omega))}
+
\|P(m)-p\|_{L^2(0,T;L^2(\Omega))}
\right).
\end{multline}
The final estimate follows from the fixed-control state error estimate.
\end{proof}
\subsection{Estimates for variational discretization of control $m$}
This subsection is devoted to deriving error estimates for the control $m$ within the fully discrete \textit{variational discretization} setting. Before proving the primary a priori bound, we first introduce a series of preliminary results that lay the foundation for the analysis.
\begin{lemma}\label{Lemma:5.7}
    Define the reduced discrete functional
\[
\begin{aligned}
        j_h(m)
        :=
        &\frac{\alpha_{\boldsymbol{u},C}}{2}
        \|U(m)-\boldsymbol{u}_C\|_{L^2(0,T; L^2(\Omega))}^{2}
        +
        \frac{\alpha_{p,C}}{2}
        \|P(m)-p_C\|_{L^2(0,T; L^2(\Omega))}^{2}  
        \\
        &+
        \frac{\gamma}{2}
        \|m\|_{L^2(0,T; L^2(\Omega))}^{2} .
\end{aligned}
\]
    For $m_1, m_2 \in \mc{M}_{ad}$ the following estimate holds:
	\begin{align}
		\left(j_{h}'(m_1) - j_{h}'(m_2),m_1-m_2\right)_{(0,T)\times \Omega} \ge \gamma  \|m_1-m_2\|_{L^2(0,T; L^2(\Omega))}^2.
	\end{align}
\end{lemma}
\begin{proof}
Let $S_h:\mathcal{M}_{ad}\to \boldsymbol{V}_{hk}\times Q_{hk}$ be the discrete control-to-state map defined by
\[
        S_h(m)=(U(m),P(m)) \in \boldsymbol{V}_{hk} \times Q_{hk}
\]
where $(U(m), P(m))$ is the solution of \eqref{eq:fullydiscrete-eqs}. Since the fully discrete state equations are linear with respect to the control variable, the map $S_h$ is affine. %
Let $m_1,m_2\in \mathcal{M}_{ad}$ be arbitrary and set
\[
        \delta m := m_1-m_2,\qquad
        \delta U := U(m_1)-U(m_2),\qquad
        \delta P := P(m_1)-P(m_2).
\]
Since the discrete control-to-state map $S_h$ is affine, its derivative is
linear and independent of the control. Hence the second variation of
$j_h$ in the direction $\delta m$ is given by
\begin{align*}
&\bigl(j_h'(m_1)-j_h'(m_2),\delta m\bigr)_{(0,T) \times \Omega}
\\
&=
\alpha_{\boldsymbol{u},C}\|\delta U\|_{L^2(0,T;L^2(\Omega))}^{2}
+
\alpha_{p,C}\|\delta P\|_{L^2(0,T;L^2(\Omega))}^{2}
+
\gamma\|\delta m\|_{L^2(0,T;L^2(\Omega))}^{2}.
\end{align*}
Since $\alpha_{\boldsymbol{u},C}\ge 0$, $\alpha_{p,C}\ge 0$, and
$\gamma>0$, we obtain
\[
\bigl(j_h'(m_1)-j_h'(m_2),m_1-m_2\bigr)_{(0,T)\times \Omega}
\ge
\gamma
\|m_1-m_2\|_{L^2(0,T;L^2(\Omega))}^{2}.
\]
\end{proof}
\begin{theorem}\label{Theorem:5.8}
	Let $(\boldsymbol{u},p,m)$ and $(U,P,M)$ be the admissible solutions of the continuous and fully discrete optimal control problems. Suppose that $(\boldsymbol{u}, p)$ and the continuous adjoint solution satisfy
	\begin{align}\label{eq:adjoint-regularity}
		p, r \in H^1(0,T; H^{1+s}(\Omega)) ,
		\qquad
		\boldsymbol{u}, \boldsymbol{w} \in H^1(0,T; \boldsymbol{H}^{1+s}(\Omega)),
	\end{align}
    and the initial data assumption \eqref{eq:initial-data-assumption} holds. 
	Then, we have the estimate 
    \begin{align*}
		\|m-M\|_{L^2(0,T;L^2(\Omega))} &\le C (\Delta t + h^s)
	\end{align*}
    with $C>0$ depending on the solution regularities in \eqref{eq:adjoint-regularity}.
\end{theorem}
\begin{proof}
Let $m$ and $M$ denote the continuous and discrete optimal controls,
respectively. We use the convention that the adjoint pressure variable
already contains the tracking weights, so that
\[
        j'(m)=\gamma m+r,
        \qquad
        j_h'(M)=\gamma M+R(M).
\]
Here $R(z)$ denotes the discrete adjoint pressure corresponding to the discrete state generated by $z \in \mathcal{M}_{ad}$.

By the strong monotonicity of the reduced discrete gradient, we have
\[
\gamma \|m-M\|_{L^2(0,T;L^2(\Omega))}^{2}
\le
\bigl(j_h'(m)-j_h'(M),m-M\bigr)_{(0,T) \times \Omega}.
\]
Using the expressions for the continuous and discrete reduced gradients,
we write
\[
\begin{aligned}
\bigl(j_h'(m)-j_h'(M),m-M\bigr)_{(0,T) \times \Omega}
& =
\bigl(\gamma m+R(m),m-M\bigr)_{(0,T) \times \Omega}
\\
&\qquad
+
\bigl(\gamma M+R(M),M-m\bigr)_{(0,T) \times \Omega}.
\end{aligned}
\]
The discrete variational inequality gives $\bigl(\gamma M+R(M),M-m\bigr)_{(0,T) \times \Omega}\le 0$. Moreover, the continuous variational inequality with the admissible test control $M$ gives
\[
\bigl(\gamma m+r,m-M\bigr)_{(0,T) \times \Omega}\le 0.
\]
Therefore,
\[
\begin{aligned}
\gamma \|m-M\|_{L^2(0,T;L^2(\Omega))}^{2}
&\le
\bigl(\gamma m+R(m),m-M\bigr)_{(0,T) \times \Omega}
\\
&=
\bigl(\gamma m+r,m-M\bigr)_{(0,T) \times \Omega}
+
\bigl(R(m)-r,m-M\bigr)_{(0,T) \times \Omega}
\\
&\le
\bigl(R(m)-r,m-M\bigr)_{(0,T) \times \Omega}.
\end{aligned}
\]
By the Cauchy--Schwarz inequality,
\[
\gamma \|m-M\|_{L^2(0,T;L^2(\Omega))}^{2}
\le
\|R(m)-r\|_{L^2(0,T;L^2(\Omega))}
\|m-M\|_{L^2(0,T;L^2(\Omega))}.
\]
If $m\neq M$, dividing by
$\|m-M\|_{L^2(0,T;L^2(\Omega))}$ yields
\[
\|m-M\|_{L^2(0,T;L^2(\Omega))}
\le
\frac{1}{\gamma}
\|R(m)-r\|_{L^2(0,T;L^2(\Omega))}.
\]
The same estimate is trivial when $m=M$. Finally, using the discrete
adjoint error estimate, $\|R(m)-r\|_{L^2(0,T;L^2(\Omega))}
\le C(h^s+\Delta t)$, 
we obtain
\[
\|m-M\|_{L^2(0,T;L^2(\Omega))}
\le C(h^s+\Delta t).
\]
This completes the proof.
\end{proof}
\begin{theorem}\label{Theorem:5.9}
Let $(\boldsymbol u,p,\bar m)$ be the continuous optimal solution and
let $(U,P,M)$ be the fully discrete optimal solution. Under the
regularity assumptions of Lemma~\ref{lemma:aux-forward-error},
Lemma~\ref{lem:adjoint-error-fixed-control}, and
Theorem~\ref{Theorem:5.8}, there exists a constant $C>0$, independent of
$h$ and $\Delta t$, such that
\[
\begin{aligned}
&\max_{1\le k\le N}
\left(
\|\boldsymbol u^k-U^k\|_{H^1(\Omega)}
+
\|p^k-P^k\|_{L^2(\Omega)}
\right)
\\
&\qquad
+
\|\boldsymbol u-U\|_{L^2(0,T;H^1(\Omega))}
+
\|p-P\|_{L^2(0,T;L^2(\Omega))}
\le
C(h^s+\Delta t).
\end{aligned}
\]
In particular, if $s=1$, then
\[
\begin{aligned}
&\max_{1\le k\le N}
\left(
\|\boldsymbol u^k-U^k\|_{H^1(\Omega)}
+
\|p^k-P^k\|_{L^2(\Omega)}
\right)
\\
&\qquad
+
\|\boldsymbol u-U\|_{L^2(0,T;H^1(\Omega))}
+
\|p-P\|_{L^2(0,T;L^2(\Omega))}
\le
C(h+\Delta t).
\end{aligned}
\]
\end{theorem}
\begin{proof}
Let $(U(\bar m),P(\bar m))$ denote the fully discrete state generated by
the continuous optimal control $\bar m$. We decompose the total state
errors as
\begin{align*}
\boldsymbol u^k-U^k
&=
\bigl(\boldsymbol u^k-U^k(\bar m)\bigr)
+
\bigl(U^k(\bar m)-U^k\bigr),
\\
p^k-P^k
&=
\bigl(p^k-P^k(\bar m)\bigr)
+
\bigl(P^k(\bar m)-P^k\bigr).
\end{align*}
The first terms are fixed-control discretization errors. By
Lemma~\ref{lemma:aux-forward-error},
\[
\begin{aligned}
&\max_{1\le k\le N}
\left(
\|\boldsymbol u^k-U^k(\bar m)\|_{H^1(\Omega)}
+
\|p^k-P^k(\bar m)\|_{L^2(\Omega)}
\right)
\\
&\qquad
+
\|\boldsymbol u-U(\bar m)\|_{L^2(0,T;H^1(\Omega))}
+
\|p-P(\bar m)\|_{L^2(0,T;L^2(\Omega))}
\le
C(h^s+\Delta t).
\end{aligned}
\]

It remains to estimate the perturbation caused by replacing $\bar m$ by
the discrete optimal control $M$. Set
\[
\zeta_{\boldsymbol u}^k:=U^k(\bar m)-U^k(M),
\qquad
\zeta_p^k:=P^k(\bar m)-P^k(M).
\]
Subtracting the two fully discrete state equations corresponding to
$\bar m$ and $M$, we obtain, for all
$(\boldsymbol v,q)\in \boldsymbol V_h\times Q_h$,
\begin{align}
a_{\boldsymbol u}(\partial_\tau\zeta_{\boldsymbol u}^k,\boldsymbol v)
+
b(\partial_\tau\zeta_p^k,\boldsymbol v)
&=0,
\\
b(q,\partial_\tau\zeta_{\boldsymbol u}^k)
-
(s_0\partial_\tau\zeta_p^k,q)_\Omega
-
a_p(\zeta_p^k,q)
&=
\frac1{\Delta t}
(\bar m-M,q)_{I_k\times\Omega}.
\end{align}
The initial data are the same for both discrete states, hence
$\zeta_{\boldsymbol u}^0=0$, $\zeta_p^0=0$.
For the right-hand side $G_m^k(q):=\frac1{\Delta t}(\bar m-M,q)_{I_k\times\Omega}$,
we have
\[
\|G_m^k\|_{a_p,*}
\le
C\Delta t^{-1/2}
\|\bar m-M\|_{L^2(I_k;L^2(\Omega))}.
\]
Consequently,
\[
\Delta t\sum_{k=0}^{N-1}\|G_m^k\|_{a_p,*}^2
\le
C\|\bar m-M\|_{L^2(0,T;L^2(\Omega))}^2.
\]
Applying the same stability argument as in the proof of
Lemma~\ref{lemma:aux-forward-error}, with zero consistency error and
with right-hand side $\bar m-M$, gives
\[
\begin{aligned}
&\max_{1\le k\le N}
\left(
\|\zeta_{\boldsymbol u}^k\|_{H^1(\Omega)}
+
\|\zeta_p^k\|_{L^2(\Omega)}
\right)
+
\|\zeta_{\boldsymbol u}\|_{L^2(0,T;H^1(\Omega))}
+
\|\zeta_p\|_{L^2(0,T;L^2(\Omega))}
\\
&\qquad
\le
C
\|\bar m-M\|_{L^2(0,T;L^2(\Omega))}.
\end{aligned}
\]
By Theorem~\ref{Theorem:5.8}, $\|\bar m-M\|_{L^2(0,T;L^2(\Omega))} \le C(h^s+\Delta t)$. 
Therefore,
\[
\begin{aligned}
&\max_{1\le k\le N}
\left(
\|\zeta_{\boldsymbol u}^k\|_{H^1(\Omega)}
+
\|\zeta_p^k\|_{L^2(\Omega)}
\right)
+
\|\zeta_{\boldsymbol u}\|_{L^2(0,T;H^1(\Omega))}
+
\|\zeta_p\|_{L^2(0,T;L^2(\Omega))}
\le
C(h^s+\Delta t).
\end{aligned}
\]

Combining the fixed-control discretization estimate and the control
perturbation estimate by the triangle inequality yields
\[
\begin{aligned}
&\max_{1\le k\le N}
\left(
\|\boldsymbol u^k-U^k\|_{H^1(\Omega)}
+
\|p^k-P^k\|_{L^2(\Omega)}
\right)
\\
&\qquad
+
\|\boldsymbol u-U\|_{L^2(0,T;H^1(\Omega))}
+
\|p-P\|_{L^2(0,T;L^2(\Omega))}
\le
C(h^s+\Delta t).
\end{aligned}
\]
For $s=1$, this gives the asserted $O(h+\Delta t)$ estimate.
\end{proof}
\section{Numerical experiments}\label{Numerical Experiments}
In this section, we showcase numerical experiments on the unit square domain $\Omega=(0,1)^2$ with the final time $T=1$ to validate the a priori error results derived in Section~\ref{A priori Error analysis}. These computations were performed using the open source finite element library FEniCS \cite{alnaes2015fenics}. For a given initial control approximation, the state system is first solved forward in time starting from the prescribed initial conditions. Subsequently, using the computed state variables, the adjoint system is solved backward in time from the terminal conditions and the control is computed using the discrete projection formula \eqref{proj} (see \cite[Algorithm~4.1]{HSAK2}). This forward--backward structure is repeated until the stopping criterion is achieved.

To calculate the right hand side body force term $\boldsymbol{f}$, we use the variational equation \eqref{eq:new-weak-n-eq1} with $(\boldsymbol{v}, q) \in \boldsymbol{V} \times Q$, instead of $\boldsymbol{V}_n \times Q_n$, to get
\begin{align}
\label{f}	\boldsymbol{f} &= -\nabla \cdot \left(\boldsymbol{\mathcal{C}}\epsilon(\boldsymbol{\dot{u}})\right) + \alpha \nabla \dot{p}= -\nabla \cdot \left(2\mu \epsilon(\boldsymbol{\dot{u}}) + \lambda \nabla \cdot(\boldsymbol{\dot{u}}) \boldsymbol{I} \right) + \alpha \nabla \dot{p}
	\\
	\nonumber &= \begin{pmatrix}
		-\mu(2\dot{u}_{1,xx} + \dot{u}_{1,yy} + \dot{u}_{2,xy}) - \lambda (\dot{u}_{1,xx} + \dot{u}_{2,yx}) + \alpha \dot{p}_{x} \\ - \mu(\dot{u}_{2,xx} + \dot{u}_{1,yx} + 2\dot{u}_{2,yy}) - \lambda (\dot{u}_{1,xy} + \dot{u}_{2,yy}) + \alpha \dot{p}_{y}
	\end{pmatrix},
\end{align} 
where $\boldsymbol{u} = (u_1,u_2)^{T}$ and the notations in the subscript denote usual derivatives. In addition to this, we use the following norms and the rate of convergence formula 
\begin{align*}
	\mathsf{e}_{\boldsymbol{u}}&:= \max_{1\le i\le N}\|\nabla(\boldsymbol{u}^i-U^i)\|_{L^2(\Omega)}, && \mathsf{e}_{p}:= \max_{1\le i\le N}\|(p^i-P^i)\|_{L^2(\Omega)}, 
    \\
	\mathsf{e}_{\boldsymbol{w}}&:= \max_{1\le i\le N}\|\nabla(\boldsymbol{w}^i-W^i)\|_{L^2(\Omega)}, && \mathsf{e}_{r}:= \max_{1\le i\le N}\|(r^i-R^i)\|_{L^2(\Omega)}, 
    \\
    \mathsf{e}_{m}&:= \|\bar{m}-m\|_{L^2(0,T; L^2(\Omega))}, &&
    \mathsf{r}_{(\cdot)} := \frac{\log\left(\mathsf{e}_{(\cdot)}/\tilde{\mathsf{e}}_{(\cdot)}\right)}{\log\left(h/\tilde{h}\right)},
\end{align*}
where $\mathsf{e}$ and $\tilde{\mathsf{e}}$ denote errors generated on two consecutive meshes of size $h$ and $\tilde{h}$, respectively. The mesh size $h=1/2^H$ and $\Delta t = h$, where $H= 2,\ldots,7$.

\begin{algorithm}
	\caption{}
	\textbf{Requirements:} Mesh, number of time steps $N$, a $N\times 1$ tolerance vector $\boldsymbol{tol} = (tol^{1},\ldots,tol^{N})^{T}$, a $N\times 1$ error vector $\boldsymbol{error}$ to represent the L2-error between the updated control $M_{new}$ 
    at each time step and an initial guess $M_{0} = (M_{0}^{1},\ldots,M_{0}^{N})^{T}$ which contains guesses for each time step.
	\begin{algorithmic}[1]\label{alg1}
		\While{$\boldsymbol{error}[k] > \boldsymbol{tol}[k]$ for any $k=0,1,\ldots,N-1$} 
		\State Solve the fully discrete state system \eqref{eq:fullydiscrete-eqs} for $k=0,1,\ldots,N-1$ using $M_{0}$.
		\State Solve the fully discrete adjoint problem \eqref{eq:fullydiscreteadj-eqs11} for $k=0,1,\ldots,N-1$ using the 
		\Statex\hspace{\algorithmicindent}solutions of fully discrete state system.
		\State Compute $M^{k}$ for $k=0,1,\ldots,N-1$ using the discrete projection formula \eqref{proj}.
		\State Store the updated control values in $M_{\text{new}}$. 
		\State Calculate the \textbf{error} between $M_{0}$ and $M_{\text{new}}$ at the corresponding time steps. 
		\State Update the control approximation as $M_{0} = M_{\text{new}}$.
		\EndWhile
	\end{algorithmic}
\end{algorithm}
In this example, we consider the given problem with boundary $\partial \Omega$ of unit square domain $\Omega$ as $\Gamma_d = \Gamma_1 \cup \Gamma_2 \cup \Gamma_4,
\Gamma_p = \Gamma_1 \cup \Gamma_2 \cup \Gamma_3 \cup \Gamma_4,\
\Gamma_t = \Gamma_3,$ and $
\Gamma_f = \emptyset$, where
\begin{align*}
	\Gamma_1 &= \{(x,y)\in \partial \Omega: y=0\}, &&\Gamma_2 = \{(x,y)\in \partial \Omega: x=1\},\\
	\Gamma_3 &= \{(x,y)\in \partial \Omega: y=1\}, &&\Gamma_4 = \{(x,y)\in \partial \Omega: x=0\}.
\end{align*} 
As analytical solutions, we consider \vspace{-2mm}
\begin{align*}
	\boldsymbol{u} &=  \begin{pmatrix}
	\sin(\pi t)\sin^2(\pi x) \sin^2(\pi y)\\
	\sin(\pi t)\sin^2(\pi x) \sin^2(\pi y)
	\end{pmatrix}, \qquad &&p = t \sin(2\pi x) \sin(2\pi y),\\
	\boldsymbol{w} &= \begin{pmatrix}
	\sin(\pi (1-t)) \sin^2(\pi x) \sin^2(\pi y)\\
	 \sin(\pi (1-t))\sin^2(\pi x) \sin^2(\pi y)
	\end{pmatrix},\qquad &&r = (1-t) \sin(2\pi x) \sin(2\pi y).
\end{align*}
We take the material parameters $s_{0} = 0.1$, $E=1$, $\nu=0.25$, $\alpha=1$, $\kappa =  \begin{pmatrix}
	3 &1\\
	1&2 
\end{pmatrix}$, control bounds $m_{a} = -0.2$ and $m_{b} = 0.2$, and $\gamma=1$. Dirichlet boundary conditions are imposed according to the exact solutions. From Fig.~\ref{FIGURE C2}, it is evident that the proposed scheme attains optimal convergence rates for displacement, pressure and control variables, with the lowest-order MINI and Taylor–Hood elements, respectively. These results confirm all the theoretical error estimates.
\begin{figure}
	\centering
	{\includegraphics[width=0.495\textwidth]{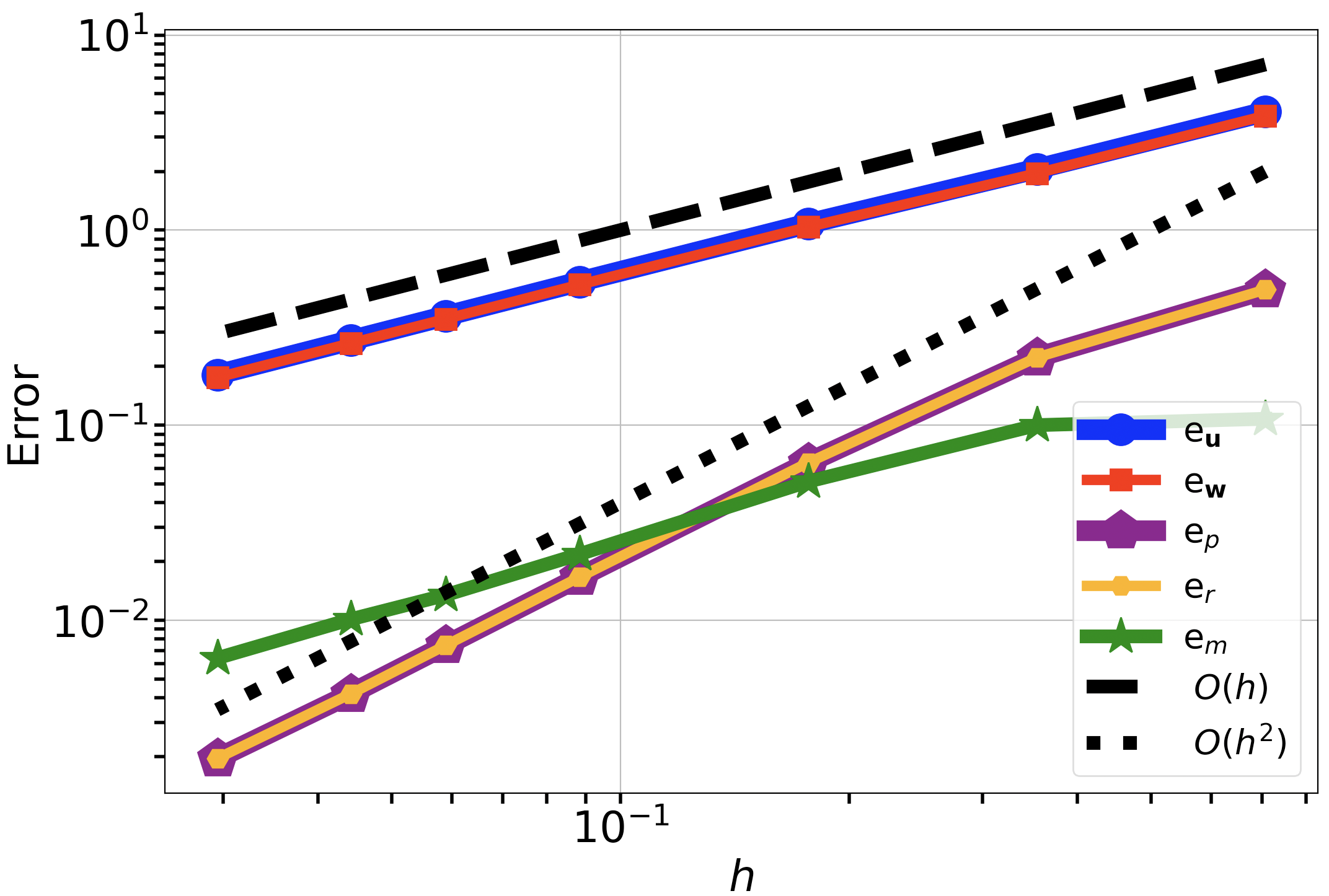}}
	{\includegraphics[width=0.495\textwidth]{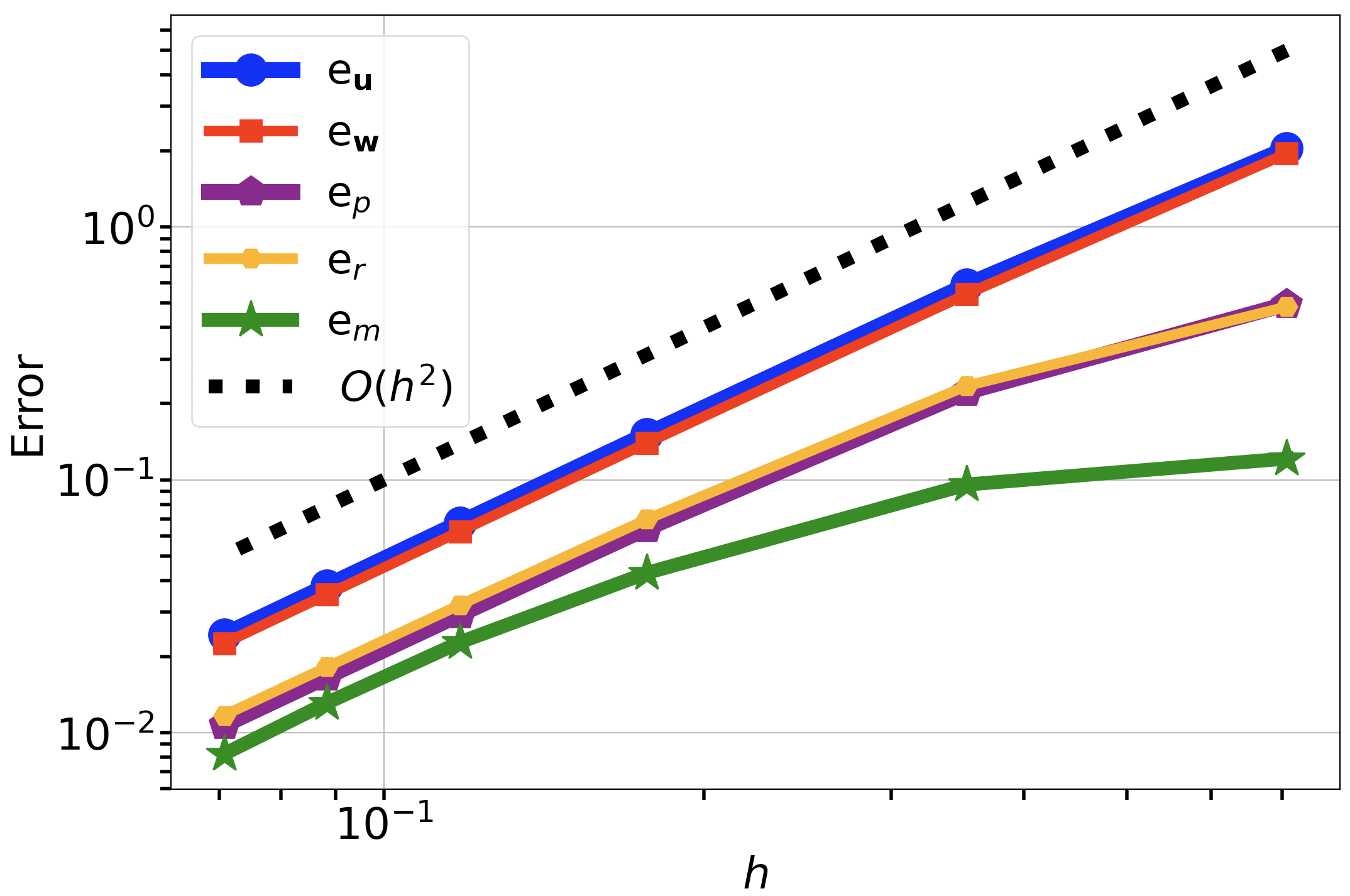}}
	\vspace{-3mm}
	\caption{Plots illustrating the convergence order of each variable computed using lowest--order MINI and Taylor–Hood elements for displacement and pressure.} \label{FIGURE C2}
\end{figure}
\section{Conclusions}\label{sec:conclusions}
In this work, we introduced a novel two-field symmetric formulation for optimal control problems governed by Biot’s model. The proposed formulation is well posed, and the existence and uniqueness of the optimal control are established. For the fully discrete setting, employing a backward Euler scheme for time discretization, we derived a priori error estimates that guarantee optimal convergence. In particular, the estimates for displacement and pressure variables are derived in $L^{\infty}(0,T;H^{1}(\Omega))$, $L^{\infty}(0,T;L^2(\Omega))$ norms and the control variable error with $L^{2}(0,T;L^{2}(\Omega))$ is approximated using a \textit{variational discretization} approach.

\section*{Acknowledgements}
The authors used AI-based tools for language editing. The authors assume responsibility for all content.

\bibliographystyle{amsplain}
\providecommand{\bysame}{\leavevmode\hbox to3em{\hrulefill}\thinspace}
\providecommand{\MR}{\relax\ifhmode\unskip\space\fi MR }
\providecommand{\MRhref}[2]{%
  \href{http://www.ams.org/mathscinet-getitem?mr=#1}{#2}
}
\providecommand{\href}[2]{#2}


\begin{thebibliography}{10}

\bibitem{alnaes2015fenics}
Martin Aln{\ae}s, Jan Blechta, Johan Hake, August Johansson, Benjamin Kehlet,
  Anders Logg, Chris Richardson, Johannes Ring, Marie~E Rognes, and Garth~N
  Wells, \emph{The fenics project version 1.5}, Archive of numerical software
  \textbf{3} (2015), no.~100.

\bibitem{biot1941general}
Maurice~A Biot, \emph{General theory of three-dimensional consolidation}, J.
  Appl. Phys. \textbf{12} (1941), no.~2, 155--164.

\bibitem{MR4410836}
Lorena Bociu and Sarah Strikwerda, \emph{Optimal control in poroelasticity},
  Appl. Anal. \textbf{101} (2022), no.~5, 1774--1796. \MR{4410836}

\bibitem{MR3504993}
Daniele Boffi, Michele Botti, and Daniele~A. Di~Pietro, \emph{A nonconforming
  high-order method for the {B}iot problem on general meshes}, SIAM J. Sci.
  Comput. \textbf{38} (2016), no.~3, A1508--A1537. \MR{3504993}

\bibitem{BMO}
Franco Brezzi and Michel Fortin, \emph{Mixed and hybrid finite element
  methods}, Springer Series in Computational Mathematics, vol.~15,
  Springer-Verlag, New York, 1991. \MR{1115205}

\bibitem{MR3022219}
Eduardo Casas and Konstantinos Chrysafinos, \emph{A discontinuous {G}alerkin
  time-stepping scheme for the velocity tracking problem}, SIAM J. Numer. Anal.
  \textbf{50} (2012), no.~5, 2281--2306. \MR{3022219}

\bibitem{MR4659441}
Aycil Cesmelioglu, Jeonghun~J. Lee, and Sander Rhebergen, \emph{Analysis of an
  embedded-hybridizable discontinuous {G}alerkin method for {B}iot's
  consolidation model}, J. Sci. Comput. \textbf{97} (2023), no.~3, Paper No.
  60, 26. \MR{4659441}

\bibitem{MR3047799}
Yumei Chen, Yan Luo, and Minfu Feng, \emph{Analysis of a discontinuous
  {G}alerkin method for the {B}iot's consolidation problem}, Appl. Math.
  Comput. \textbf{219} (2013), no.~17, 9043--9056. \MR{3047799}

\bibitem{MR3907413}
Guosheng Fu, \emph{A high-order {HDG} method for the {B}iot's consolidation
  model}, Comput. Math. Appl. \textbf{77} (2019), no.~1, 237--252. \MR{3907413}

\bibitem{MR2122182}
M.~Hinze, \emph{A variational discretization concept in control constrained
  optimization: the linear-quadratic case}, Comput. Optim. Appl. \textbf{30}
  (2005), no.~1, 45--61. \MR{2122182}

\bibitem{MR4405491}
Arbaz Khan and Pietro Zanotti, \emph{A nonsymmetric approach and a
  quasi-optimal and robust discretization for the {B}iot's model}, Math. Comp.
  \textbf{91} (2022), no.~335, 1143--1170. \MR{4405491}

\bibitem{Kim1999549}
Jun-Mo Kim and Richard~R. Parizek, \emph{Three-dimensional finite element
  modelling for consolidation due to groundwater withdrawal in a desaturating
  anisotropic aquifer system}, Internat. J. Numer. Anal. Methods Geomech.
  \textbf{23} (1999), no.~6, 549 – 571.

\bibitem{MR2177147}
Johannes Korsawe and Gerhard Starke, \emph{A least-squares mixed finite element
  method for {B}iot's consolidation problem in porous media}, SIAM J. Numer.
  Anal. \textbf{43} (2005), no.~1, 318--339. \MR{2177147}

\bibitem{MR3803860}
Jeonghun~J. Lee, \emph{Robust three-field finite element methods for {B}iot's
  consolidation model in poroelasticity}, BIT \textbf{58} (2018), no.~2,
  347--372. \MR{3803860}

\bibitem{MR3590654}
Jeonghun~J. Lee, Kent-Andre Mardal, and Ragnar Winther, \emph{Parameter-robust
  discretization and preconditioning of {B}iot's consolidation model}, SIAM J.
  Sci. Comput. \textbf{39} (2017), no.~1, A1--A24. \MR{3590654}

\bibitem{MR4636155}
Hao Liang and Hongxing Rui, \emph{The nonconforming locking-free virtual
  element method for the {B}iot's consolidation model in poroelasticity},
  Comput. Math. Appl. \textbf{148} (2023), 269--281. \MR{4636155}

\bibitem{malandrino2019poroelasticity}
Andrea Malandrino and Emad Moeendarbary, \emph{Poroelasticity of living
  tissues}, Encyclopedia of biomedical engineering (2019), 238--245.

\bibitem{mccormack2020modeling}
Kimberly McCormack, Marc~A Hesse, Timothy Dixon, and Rocco Malservisi,
  \emph{Modeling the contribution of poroelastic deformation to postseismic
  geodetic signals}, Geophysical Research Letters \textbf{47} (2020), no.~8.

\bibitem{phillips2008coupling}
Phillip~Joseph Phillips and Mary~F Wheeler, \emph{A coupling of mixed and
  discontinuous galerkin finite-element methods for poroelasticity}, Comput.
  Geosci. \textbf{12} (2008), no.~4, 417--435.

\bibitem{phillips2009}
\bysame, \emph{Overcoming the problem of locking in linear elasticity and
  poroelasticity: an heuristic approach}, Comput. Geosci. \textbf{13} (2009),
  no.~1, 5--12.

\bibitem{MR3606362}
B\'eatrice Rivi\`ere, Jun Tan, and Travis Thompson, \emph{Error analysis of
  primal discontinuous {G}alerkin methods for a mixed formulation of the {B}iot
  equations}, Comput. Math. Appl. \textbf{73} (2017), no.~4, 666--683.
  \MR{3606362}

\bibitem{HSAK2}
Harpal Singh and Arbaz Khan, \emph{Conforming/non-conforming mixed finite
  element methods for optimal control of velocity-vorticity-pressure
  formulation for the {O}seen problem with variable viscosity}, Comput. Math.
  Appl. \textbf{198} (2025), 59--92.

\bibitem{Swan200325}
Colby~C. Swan, R.S. Lakes, R.A. Brand, and K.J. Stewart,
  \emph{Micromechanically based poroelastic modeling of fluid flow in haversian
  bone}, J. Biomech. Eng. \textbf{125} (2003), no.~1, 25 – 37.

\bibitem{MR4221326}
Xialan Tang, Zhibin Liu, Baiju Zhang, and Minfu Feng, \emph{On the locking-free
  three-field virtual element methods for {B}iot's consolidation model in
  poroelasticity}, ESAIM Math. Model. Numer. Anal. \textbf{55} (2021),
  S909--S939. \MR{4221326}

\bibitem{WANGEN2016486}
Magnus Wangen, Sarah Gasda, and Tore Bj{\o}rnar{\aa}, \emph{Geomechanical
  consequences of large-scale fluid storage in the utsira formation in the
  north sea}, Energy Proc. \textbf{97} (2016), 486--493.

\bibitem{MR2273503}
Tessa Weinstein and Lynn~S. Bennethum, \emph{On the derivation of the transport
  equation for swelling porous materials with finite deformation}, Internat. J.
  Engrg. Sci. \textbf{44} (2006), no.~18-19, 1408--1422. \MR{2273503}

\bibitem{MR2644299}
Zhaojie Zhou and Ningning Yan, \emph{The local discontinuous {G}alerkin method
  for optimal control problem governed by convection diffusion equations}, Int.
  J. Numer. Anal. Model. \textbf{7} (2010), no.~4, 681--699. \MR{2644299}

\end{thebibliography}
\end{document}